\documentclass[11pt,a4paper]{article}
\title{Sub-Feller Semigroups Generated by Pseudodifferential Operators on Symmetric Spaces of Noncompact Type}
\author{Rosemary Shewell Brockway}
\date{\today}
\usepackage[margin=3cm]{geometry}
\usepackage[round]{natbib}
\usepackage{amsmath,amssymb,amsthm,mathtools,tikz,mathrsfs}
\numberwithin{equation}{section}
\newtheorem{thm}{Theorem}[section]
\newtheorem{lem}[thm]{Lemma}
\newtheorem{prop}[thm]{Proposition}
\newtheorem{cor}[thm]{Corollary}
\newtheorem{ass}[thm]{Assumption}
\newtheorem{asss}[thm]{Assumptions}
\theoremstyle{definition}
\newtheorem{mydef}[thm]{Definition}
\newtheorem{eg}[thm]{Example}
\theoremstyle{remark}
\newtheorem{rmk}[thm]{Remark}
\newtheorem{rmks}[thm]{Remarks}
\newtheorem*{claim}{Claim}
\newtheorem*{pf}{Proof of Claim}
\newtheorem*{ack}{Acknowledgement}
\usepackage[hide links]{hyperref}

\DeclareMathOperator{\Dom}{Dom}

\DeclareMathOperator{\tr}{tr}

\DeclareMathOperator{\ind}{\bf 1}
\DeclareMathOperator{\Ran}{Ran}
\DeclareMathOperator{\Ad}{Ad}
\DeclareMathOperator{\ad}{ad}
\DeclareMathOperator{\hcc}{\bf c}

\DeclareMathOperator{\a*}{\mathfrak{a}^\ast}
\DeclareMathOperator{\al}{\mathfrak{a}}
\DeclareMathOperator{\aC*}{\mathfrak{a}_\mathbb{C}^\ast}
\DeclareMathOperator{\Supp}{Supp}

\begin{document}
\maketitle
%
\begin{abstract}
We consider global pseudodifferential operators on symmetric spaces of noncompact type, defined using spherical functions. The associated symbols have a natural probabilistic form that extend the notion of the characteristic exponent appearing in Gangolli's L\'evy--Khinchine formula to a function of two variables. The Hille--Yosida--Ray theorem is used to obtain conditions on such a symbol so that the corresponding pseudodifferential operator has an extension that generates a sub-Feller semigroup, generalising existing results for Euclidean space.
\end{abstract}

\vspace{5pt}


\vspace{5pt}


\vspace{5pt}
%
\section{Introduction}
Pseudodifferential operator theory is a powerful tool in the study of Feller--Markov processes on Euclidean space (see for example \citet{KnopSchilWang,Schil_ConsFeller,Schil_ConsPSDOs,Schil_GrowthHolder,SchilSchn}, or \citet{bsw} \S 5 for a summary). Primarily developed by Niels Jacob and collaborators (see e.g.~\citet{jacob93,jacob94,hoh}), this framework characterises sub-Feller semigroups and their generators as pseudodifferential operators ($\Psi$DOs) acting on $C_0(\mathbb{R}^d)$, the Banach space of continuous, real-valued functions on $\mathbb{R}^d$ that vanish at infinity. The associated symbols capture many properties of the sub-Feller processes and semigroups they represent, generalising the well-established relationship between L\'evy processes and their characteristic exponents given by the L\'evy--Khinchine formula. The key difference is that the Feller--Markov symbols typically depend on two variables instead of one --- in the case of Feller processes, this is sometimes described as the L\'evy characteristics having gained spatial dependence.

Manifold-valued Feller--Markov processes have also attracted interest in recent years (see \citet{elworthy,hsu} and \citet{kunita_JDs} \S 7 for excellent summaries), though the absence of a global harmonic analysis on general manifolds has so far prevented a $\Psi$DO approach. Lie groups and symmetric spaces come with their own harmonic analysis, however, in the form of the spherical transform (see \citet{HCI,HCII} and \citet{helg1,helg2}). A natural question to ask then is to what extent a $\Psi$DO-based approach can be applied to the study of sub-Feller processes on Lie groups and symmetric spaces. Much work has already been done in this area, especially in the ``constant coefficient'' case of L\'evy processes, in which the symbol depends only on its second argument. Here, the symbols are given by Gangolli's L\'evy--Khinchine formula (\citet{gangolli} Theorem 6.2), a direct analogue to the classical result. The first paper to use probabilistic $\Psi$DO methods on a Lie group was \citet{AppCohen}, where $\Psi$DOs are used to study L\'evy processes on the Heisenberg group. Pseudodifferential operator representations of semigroups and generators have also been found for L\'evy processes on a general Lie group --- see \citet{App_InfDivCompact,App_PSDO_Compact} for the compact case and \citet{App_AusMathSoc} Section 5 for the noncompact case. For Feller processes, $\Psi$DO representations have been found when the Lie group or symmetric space is compact --- see \citet{AppNgan_PMPLie,AppNgan_PMPSS}. 

This paper seeks to develop a more general theory of $\Psi$DOs for symmetric spaces of noncompact type, and apply it to seek conditions on a symbol so that the corresponding $\Psi$DO extends to the generator of some sub-Feller process. For the $\mathbb{R}^d$ case, this question has been studied thoroughly in \citet{jacob93,jacob94}, as well as \citet{hoh} Chapter 4. 

The spherical transform enjoys many of the same properties as the Fourier transform on $\mathbb{R}^d$, and we find that several of the arguments in \citet{jacob94} and \citet{hoh} generalise directly to the symmetric space setting. However, a direct transcription of Jacob and Hoh's arguments is certainly not possible. One notable difference, for example, is the use of directional derivatives in the condition (2.2) of \citet{jacob94}, which on a manifold would depend on a choice of coordinate chart. Even on a symmetric space, there was no straightforward analogue for this condition, and a different approach is needed. We take a more operator-theoretic approach, and replace each of the partial derivatives $\frac{\partial}{\partial x_i}$ ($i=1,\ldots,d$) with the fractional Laplacian $\sqrt{-\Delta}$. This is an exciting object to work with, and we found it to be far more compatible with our harmonic analytical approach.

The structure of this paper will be the following. Section \ref{sec:prelim} presents a summary of necessary concepts and results from harmonic analysis on symmetric spaces, and introduces  the system of symbols and $\Psi$DOs that will be used later on. We also introduce here the spherical anisotropic Sobolev spaces, a generalisation of the anisotropic Sobolev spaces first considered by Niels Jacob --- see e.g.~\citet{jacob93,jacob94}.

In Section \ref{sec:GangOps&HYR}, we consider the Hille--Yosida--Ray theorem (see Theorem \ref{thm:hyr}), and build on the work of \citet{AppNgan_PMPLie,AppNgan_PMPSS}, introducing a class of operators we will call Gangolli operators, which satisfy all but one of the conditions of Hille--Yosida--Ray. We prove that Gangolli operators are $\Psi$DOs in the sense of Section \ref{subsec:Ops&Symbs}, and derive a formula for their symbols (Theorem \ref{thm:psdo}).

Section \ref{sec:J&H} is concerned with seeking sufficient conditions for a Gangolli operator $q(\sigma,D)$ to extend to the generator a sub-Feller semigroup. Informed by the work of the previous section, this amounts to finding conditions that ensure 
	\[\overline{\Ran(\alpha+q(\sigma,D))}=C_0(K|G|K)\]
for some $\alpha>0$ (see Theorem \ref{thm:hyr} (\ref{dense})). This section perhaps most closely follows the approach of \citet{jacob94} and \citet{hoh} Chapter 4, and where proofs are similar to these sources, we omit detail, and instead aim to emphasise what is different about the symmetric space setting. Full proofs may be found in my PhD thesis \citep{RSB_thesis}.

In Section \ref{sec:EGs} we present a large class of examples of symbols that satisfy the conditions found in Sections \ref{sec:GangOps&HYR} and \ref{sec:J&H}. 
\\

{\bf Notation.} For a topological space $X$, $\mathcal{B}(X)$ will denote the Borel $\sigma$-algebra associated with $X$, and $B_b(X)$ the space of bounded, Borel functions from $X\to\mathbb{R}$, a Banach space with respect to the supremum norm. If $X$ is a locally compact Hausdorff space, then we write $C_0(X)$ for the closed subspace of $B_b(X)$ consisting of continuous functions vanishing at infinity, and $C_c(X)$  for the  dense subspace of compactly supported continuous functions. If $X$ is a smooth manifold and $M\in\mathbb{N}\cup\{\infty\}$, then we write $C^M_c(X)$ for the space of compactly supported $M$-times continuously differentiable functions on $X$.
%
\section{Preliminaries}\label{sec:prelim}
Let $M$ be a Riemannian symmetric space. By Theorem 3.3 of \citet{helg1}, pp.~208, $M$ is diffeomorphic to a homogeneous space $G/K$, where $G$ is a connected Lie group and $K$ is a compact subgroup of $G$. Moreover, for some nontrivial involution $\Theta$ on $G$,
	\[G^\Theta_0\subseteq K\subseteq G^\Theta,\]
where $G^\Theta$ is the fixed point set of $\Theta$, and $G^\Theta_0$ is the identity component of $G^\Theta$. Let $\mathfrak{g}$ and $\mathfrak{k}$ denote the Lie algebras of $G$ and $K$, respectively. Note that $\mathfrak{k}$ is the $+1$ eigenspace of the differential $\theta:=d\Theta$; let $\mathfrak{p}$ denote the $-1$ eigenspace. In fact, $\theta$ is a Cartan involution on $\mathfrak{g}$, and the corresponding Cartan decomposition is
	\begin{equation}\label{eq:cartan}
	\mathfrak{g} = \mathfrak{p} \oplus \mathfrak{k}.
	\end{equation}
Let $B$ denote the Killing form of $G$, defined for each for all $X,Y\in\mathfrak{g}$ by $B(X,Y) = \tr(\ad X\ad Y)$. Assume that $(G,K)$ is of noncompact type, so that $B$ negative definite on $\mathfrak{k}$ and positive definite on $\mathfrak{p}$. Since $B$ is nondegenerate, $G$ is semisimple. 

Fix an $\Ad(K)$-invariant inner product $\langle\cdot,\cdot\rangle$ on $\mathfrak{g}$, with respect to which (\ref{eq:cartan}) is an orthogonal direct sum. The Riemannian structure of $M\cong G/K$ is induced by the restriction of $\langle\cdot,\cdot\rangle$ to $\mathfrak{p}$. 

There is a one to one correspondence between functions on $G/K$ and $K$-right-invariant functions on $G$, we denote both by $\mathcal{F}(G/K)$. Similarly, we identify $K$-invariant functions on $G/K$ with $K$-bi-invariant functions on $G$, and denote both by $\mathcal{F}(K|G|K)$. Similar conventions will be used to denote standard subspaces of $\mathcal{F}(G/K)$ and $\mathcal{F}(K|G|K)$; for example $C(K|G|K)$ will denote both the space of continuous, $K$-invariant functions on $G/K$, and the space of continuous, $K$-bi-invariant functions on $G$.

Equip $G$ with Haar measure, and for $p\geq 1$ let $L^p(G)$ denote the corresponding $L^p$ of real-valued functions. $L^p(K|G|K)$ will denote the closed subspace of $L^p(G)$ consisting of $K$-bi-invariant elements. Since $G$ is unimodular, Haar measure is translation invariant, and can be projected onto the coset spaces $G/K$ and $K|G|K$ in a well-defined way. We continue to identify $K$-bi-invariant functions on $G$ with functions on $K|G|K$, as well as with $K$-invariant functions on $G/K$, in this $L^p$ setting.
%
\subsection{Harmonic Analysis on Symmetric Spaces of Noncompact Type}
For a thorough treatment of this topic, see \citet{helg1,helg2}. Let ${\bf D}(G)$ denote the set of all left invariant differential operators on $G$, and let ${\bf D}_K(G)$ denote the subspace of those operators that are also $K$-right-invariant. A mapping $\phi:G\to\mathbb{C}$ is called a \emph{spherical} if it is $K$-bi-invariant, satisfies $\phi(e) =1$, and is a simultaneous eigenfunction of every element of ${\bf D}_K(G)$. 

Fix an Iwasawa decomposition $G=NAK$, where $N$ is a nilpotent Lie subgroup of $G$, and $A$ is Abelian. Let $\mathfrak{n}$ and $\mathfrak{a}$ denote respectively the Lie algebras of $N$ and $A$. For each $\sigma\in G$, let $A(\sigma)$ denote the unique element of $\al	$ such that $\sigma\in Ne^{A(\sigma)}K$. Harish-Chandra's integral formula states that every spherical function on $G$ takes the form
	\begin{equation}\label{eq:H-C}
	\phi_\lambda(\sigma) = \int_Ke^{(\rho+i\lambda)(A(k\sigma))}dk, \hspace{20pt} \forall\sigma\in G,
	\end{equation}
for some $\lambda\in\aC*$. Moreover, $\phi_\lambda=\phi_{\lambda'}$ if an only if $s(\lambda)=\lambda'$ for some element $s$ of the Weyl group $W$. A spherical function $\phi_\lambda$ is positive definite if and only if $\lambda\in\a*$. 

The \emph{spherical transform} of a function $f\in L^1(K|G|K)$ is the function $\hat{f}:\a*\to\mathbb{C}$ given by
	\begin{equation}\label{eq:SphTr}
	\hat{f}(\lambda) = \int_G\phi_{-\lambda}(\sigma)f(\sigma)d\sigma, \hspace{20pt} \forall\lambda\in\a*.
	\end{equation}
Similarly, given a finite Borel measure $\mu$ on $G$, the spherical transform of $\mu$ is the mapping $\hat{\mu}:\a*\to\mathbb{C}$ given by
	\[\hat{\mu}(\lambda) = \int_G\phi_{-\lambda}(\sigma)\mu(d\sigma).\]
The spherical transform enjoys many useful properties, the most powerful being that it defines an isomorphism of the Banach convolution algebra $L^1(K|G|K)$ with the space $L^1(\a*,\omega)^W$ of Weyl group invariant elements of $L^1(\a*,\omega)$. The Borel measure $\omega$ is called \emph{Plancherel measure}, and is given by
	\[\omega(d\lambda) = |\hcc(\lambda)|^{-2}d\lambda,\]
where $\hcc$ denotes Harish-Chandra's $\hcc$-function. According to the spherical inversion formula, for all $f\in C_c^\infty(K|G|K)$ and all $\sigma\in G$,
	\begin{equation}\label{eq:SphInv}
	f(\sigma) = \int_{\a*}\phi_\lambda(\sigma)\hat{f}(\lambda)\omega(d\lambda).
	\end{equation}
There is also a version of Plancherel's identity for the spherical transform, namely
	\begin{equation}\label{eq:Plancherel}
	\Vert f\Vert_{L^2(K|G|K)} = \Vert\hat{f}\Vert_{L^2(\a*,\omega)}, \hspace{20pt} \forall f\in C_c^\infty(K|G|K).
	\end{equation}
Let $L^2(\a*,\omega)^W$ denote the subspace of $L^2(\a*,\omega)$ consisting of $W$-invariants. Then the image of $C_c^\infty(K|G|K)$ under spherical transformation is a dense subspace of $L^2(\a*,\omega)^W$, and as such the spherical transform extends to an isometric isomorphism between the Hilbert spaces $L^2(K|G|K)$ and $L^2(\a*,\omega)^W$. For more details, see for example \citet{helg2} Chapter IV \S~7.3, pp.~454.

Similarly to classical Fourier theory, the most natural setting for the spherical transform is Schwarz space. A function $f\in C^\infty(G)$ is called \emph{rapidly decreasing} if
	\begin{equation}\label{eq:schwartz}
	\sup_{\sigma\in G}(1+|\sigma|)^q\phi_0(\sigma)^{-1}(Df)(\sigma) < \infty, \hspace{20pt} \forall D\in{\bf D}(G), \;q\in\mathbb{N}\cup\{0\},
	\end{equation}
where $|\sigma|$ denotes the geodesic distance on $G/K$ from $o:=eK$ to $\sigma K$. Equivalently, if $\sigma=e^Xk$, where $X\in\mathfrak{p}$, then $|\sigma|=\Vert X\Vert$ --- see \citet{GangVar} pp.167 for more details.

The \emph{($K$-bi-invariant) Schwarz space} $\mathcal{S}(K|G|K)$ is the Fr\'echet space comprising of all rapidly decreasing, $K$-bi-invariant functions $f\in C^\infty(G)$, together with the family of seminorms given by the left-hand side of (\ref{eq:schwartz}). By viewing the spaces $\al$ and $\a*$ as finite dimensional vector spaces, we also consider the classical Schwartz spaces $\mathcal{S}(\a*)$ and $\mathcal{S}(\al)$, as well as $W$-invariant subspaces $\mathcal{S}(\al)^W$ and $\mathcal{S}(\a*)^W$. The Euclidean Fourier transform
	\begin{equation}\label{eq:EucF}
	\mathscr{F}(f)(\lambda) = \int_{\al}e^{-i\lambda(H)}f(H)dH, \hspace{20pt} \forall f\in\mathcal{S}(\al),\lambda\in\a*
	\end{equation}
defines a topological isomorphism between the spaces $\mathcal{S}(\al)^W$ and $\mathcal{S}(\a*)^W$ in the usual way. Given $f\in\mathcal{S}(K|G|K)$ and $H\in\al$, the \emph{Abel transform} is defined by
	\[\mathscr{A}f (H) = e^{\rho(H)}\int_Nf((\exp H)n)dn.\]
The Abel transform is fascinating in its own right, and we refer to \citet{sawyer} for more information. However, for our purposes we are mainly interested in its role in the following:
\begin{thm}\label{thm:ComDiag}
Writing $\mathscr{H}$ for the spherical transform, the diagram 
\begin{center}
\begin{tikzpicture}[scale=2]
\node (1) at (0,1){$\mathcal{S}(\a*)^W$};
\node (2) at (1,0){$\mathcal{S}(\al)^W$};
\node (3) at (-1,0){$\mathcal{S}(K|G|K)$};
\draw[->] (2) -- (1);
\draw[->] (3) -- (1);
\draw[->] (3) -- (2);
\node at (.63,.63) {$\mathscr{F}$};
\node at (-.63,.63) {$\mathscr{H}$};
\node at (0,-.2) {$\mathscr{A}$};
\end{tikzpicture}
\end{center}
commutes, up to normalizing constants. Each arrow describes an isomorphism of Fr\'echet algebras.
\end{thm}
This result will be extremely useful in later sections, especially when proving Theorem \ref{thm:strong}. For more details, see Proposition 3 in \citet{anker90}, \citet{GangVar} page 265, and \citet{helg2} pp.~450. 

%
\subsection{Probability on Lie Groups and Symmetric spaces}
We summarise a few key notions from probability theory on Lie groups and symmetric spaces. Sources for this material include \citet{LiaoWang} and \citet{Liao_LevyLie,Liao_InvMark}.

Fix a probability space $(\Omega,\mathcal{F},P)$. Just as with functions on $G$ and $G/K$, we may view stochastic processes on $G/K$ as projections of processes on $G$ whose laws are $K$-right invariant. Let $Y=(Y(t),t\geq 0)$ a stochastic process taking values on $G$. The random variables 
	\[Y(s)^{-1}Y(t), \hspace{20pt} 0\leq s\leq t,\]
are called the \emph{increments} of $Y$. Equipped with its natural filtration $\{\mathcal{F}^Y_t,t\geq 0\}$, $Y$ is said to have \emph{independent increments} if for all $t>s\geq 0$, $Y(s)^{-1}Y(t)$ is independent of $\mathcal{F}^X_s$, and \emph{stationary increments} if
	\[Y(s)^{-1}Y(t) \sim Y(0)^{-1}Y(t-s) \hspace{20pt} \forall t>s\geq 0.\]
A process $Y=(Y(t),t\geq 0)$ on $G$ is \emph{stochastically continuous} if, for all $s\geq 0$ and all $B\in\mathcal{B}(G)$ with $e\notin B$,
	\[\lim_{t\rightarrow s}P(Y(s)^{-1}Y(t)\in B)=0.\]
A stochastically continuous process $Y$ on $G$ with stationary and independent increments is called a \emph{L\'evy process} on $G$. A process on $G/K$ is called a L\'evy process if it is the projection of a L\'evy process on $G$, under the canonical surjection $\pi:G\mapsto G/K$. L\'evy processes on $G/K$ correspond precisely to the $G$-invariant Feller processes on $G/K$. The proof of this is similar to the well-known result for $\mathbb{R}^d$-valued L\'evy processes.

The \emph{convolution product} of two Borel measures $\mu_1,\mu_2$ on $G$ is defined for each $B\in\mathcal{B}(G)$ by
	\begin{equation}\label{eq:convmeasG}
	(\mu_1\ast\mu_2)(B) = \int_G\int_G\ind_B(\sigma\tau)\mu_1(d\sigma)\mu_2(d\tau).
	\end{equation}
Note that since $G$ is semisimple, it is unimodular, and hence this operation is commutative. It is also clear from the definition that $\mu_1\ast\mu_2$ is $K$-bi-invariant whenever $\mu_1$ and $\mu_2$ are.

\begin{mydef}\label{mydef:convsemigpG}
A family $(\mu_t,t\geq 0)$ of finite Borel measures on $G$ will be called a \emph{convolution semigroup (of probability measures)} if 
\begin{enumerate}
\item\label{probmeas} $\mu_t(G)=1$ for all $t\geq 0$, 
\item $\mu_{s+t} = \mu_s\ast\mu_t$ for all $s,t\geq 0$, and
\item $\mu_t\rightarrow\mu_0$ weakly as $t\rightarrow 0$.
\end{enumerate}
\end{mydef}
Note that $\mu_0$ must be an idempotent measure, in the sense that $\mu_0\ast\mu_0=\mu_0$. By Theorem 1.2.10 on page 34 of \citet{Heyer_ProbMeasonLCGs}, $\mu_0$ must coincide with Haar measure on a compact subgroup of $G$. We we will frequently take $\mu_0$ to be normalised Haar measure on $K$, so that the image of $\mu_0$ after projecting onto $G/K$ is $\delta_o$, the delta mass at $o:=eK$. 

One may also define convolution of measures on $G/K$, and convolution semigroups on $G/K$ are defined analogously --- see \citet{Liao_InvMark} Section 1.3 for more details. In fact, the projection map $\pi:G\to G/K$ induces a bijection between the set of all convolution semigroups on $G/K$ and the set of all $K$-bi-invariant convolution semigroups on $G$ --- see \citet{Liao_InvMark} Propositions 1.9 and 1.12, pp.~11--13. We henceforth identify these two sets, but generally opt to perform calculations using objects defined on $G$, for simplicity.
 
Let $Y$ be a L\'evy process on $G$, and for each $t\geq 0$, let $\mu_t$ denote the law of $Y(0)^{-1}Y(t)$. By \citet{Liao_InvMark} Theorem 1.7, pp.~8, $(\mu_t,t\geq 0)$ is a convolution semigroup of probability measures on $G$.

\begin{mydef}\label{mydef:convsemiass}
We call $(\mu_t,t\geq 0)$ the \emph{convolution semigroup associated with $X$}. 
\end{mydef}

Let $Y$ be a L\'evy process on $G/K$, and $X$ a  L\'evy process on $G$ for which $Y=\pi(X)$. Let $(p_t,t\geq 0)$ and $(q_t,t\geq 0)$ denote the transition probabilities of $Y$ and $X$, respectively. Then for all $t\geq 0$, $\sigma\in G$ and $A\in\mathcal{B}(G/K)$,
	\[p_t(\sigma K,A) = \mathbb{P}\left.\left(\pi\big(X(t)\big)\in A\right|\pi\big(X\big)=\sigma K\right) = q_t\left(\sigma,\pi^{-1}(A)\right).\]
In particular, the prescription
	\[\nu_t := p_t(o,\cdot), \hspace{20pt} \forall t\geq 0\]
defines a convolution semigroup $(\nu_t,t\geq 0)$ on $G/K$. By \citet{Liao_InvMark} Proposition 1.12, pp.~13, $(\nu_t,t\geq 0)$ is $K$-invariant, and there is a $K$-bi-invariant convolution semigroup $(\mu_t,t\geq 0)$ on $G$ for which
	\begin{equation}\label{eq:nu_t}
	\nu_t=\mu_t\circ\pi^{-1}, \hspace{20pt} \forall t\geq 0.
	\end{equation}
It may be tempting to think that $(\mu_t,t\geq 0)$ should be the convolution semigroup of $X$. In fact, this is not the case: if it were, then we would have $\mu_0=\delta_0$, which is not a $K$-bi-invariant measure on $G$. However, if we denote the convolution semigroup of $X$ by $(\mu^e_t,t\geq 0)$, and normalised Haar measure on $K$ by $\rho_K$, then by \citet{Liao_InvMark} Theorem 3.14, pp.~88, 
	\[\mu_t:=\rho_K\ast\mu^e_t, \hspace{20pt} \forall t\geq 0\]
is a suitable choice for the $K$-bi-invariant convolution semigroup $(\mu_t,t\geq 0)$ on $G$, for which (\ref{eq:nu_t}) is satisfied. In particular, $\mu_0=\rho_K$.

In this way, L\'evy processes on $G/K$ may be understood through the study of $K$-bi-invariant convolution semigroups on $G$. The corresponding L\'evy processes on $G$ are called \emph{$K$-bi-invariant L\'evy processes}. For such a process $X$, with $K$-bi-invariant convolution semigroup $(\mu_t,t\geq 0)$, the \emph{Hunt semigroup} $(T_t,t\geq 0)$ of $(\mu_t,t\geq 0)$) is given by
	\begin{equation}\label{eq:Huntsemi}
	T_tf(\sigma) = \int_Gf(\sigma\tau)\mu_t(d\tau) \hspace{20pt} \forall f\in B_b(G), \;\sigma\in G. 
	\end{equation}
Note that since $\mu_0=\rho_K$, we have $T_0=I$. In fact, $(T_t,t\geq 0)$ forms a strongly continuous operator semigroup on $C_0(G/K)$, and the restriction of each $T_t$ to $C_0(K|G|K)$ yields a strongly continuous semigroup on $C_0(K|G|K)$. $(T_t,t\geq 0)$ is a left invariant Feller semigroup in each of these cases (see \citet{Ngan} pp.~82--83).

Restricting to the $K$-bi-invariant functions in this way will be advantageous, as we have the spherical transform at our disposal. As an early application of this, we prove the following useful eigenvalue relation for the Hunt semigroup of a $K$-bi-invariant convolution semigroup.
\begin{prop}\label{prop:T_tɸ_λ(𝛔)}
Let $(\mu_t,t\geq 0)$ be a $K$-bi-invariant convolution semigroup  on $G$, and let $(T_t,t\geq 0)$ denote the the restriction to $C_0(K|G|K)$ of the Hunt semigroup associated with $(\mu_t,t\geq 0)$. Then for all $t\geq 0$, $\lambda\in\a*$ and $\sigma\in G$,
	\[T_t\phi_\lambda(\sigma) = \hat\mu_t(\lambda)\phi_\lambda(\sigma).\]
\end{prop}
\begin{proof}
Let $t\geq 0$, $\lambda\in \a*$ and $\sigma\in G$. Observe that since each $\mu_t$ is invariant under all translations by $K$,
	\[T_t\phi_\lambda(\sigma) = \int_G\phi_\lambda(\sigma k\tau)\mu_t(d\tau)\]
for each $k\in K$. Integrating over $K$ and applying a Fubini argument,
	\[T_t\phi_\lambda(\sigma) = \int_K\int_G\phi_\lambda(\sigma k\tau)\mu_t(d\tau)dk =  \int_G\int_K\phi_\lambda(\sigma k\tau)dk\mu_t(d\tau).\]
We can now apply the beautiful integral formula for spherical functions,
	\begin{equation}\label{eq:SphIntForm}
	\phi_\lambda(\sigma)\phi_\lambda(\tau) = \int_K\phi_\lambda(\sigma k\tau)dk
	\end{equation}
(c.f~\citet{helg2} pp.~400--402), to conclude
	\[T_t\phi_\lambda(\sigma) = \int_G\phi_\lambda(\sigma)\phi(\tau)\mu_t(d\tau) = \phi_\lambda(\sigma)\hat\mu_t(\lambda),\]
as desired.
\end{proof}

The infinitesimal generator of a L\'evy process $Y$ on $G$ is given by the celebrated Hunt formula (\cite{hunt} Theorem 5.1). We describe a version of this next, specialising to the $K$-bi-invariant case most relevant to our work on symmetric spaces. We first introduce a local coordinate system on $G$, defined in terms of the orthogonal decomposition (\ref{eq:cartan}).

\begin{mydef}\label{mydef:expcoords}
Let $X_1,\ldots,X_l$ be a basis of $\mathfrak{g}$, ordered so that $X_1,\ldots,X_d$ is a basis of $\mathfrak{p}$. A collection $\{x_1,\ldots,x_l\}$ of smooth functions of compact support is called a \emph{system of exponential coordinate functions} if there is a neighbourhood $U$ of $e$ for which
	\begin{equation}\label{eq:coordnbd}
	\sigma = \exp\left(\sum_{i=1}^lx_i(\sigma)X_i\right) \hspace{20pt} \forall \sigma\in U.
	\end{equation}
\end{mydef}
The $x_i$ may be chosen so as to be $K$-right-invariant for $i=1,\ldots, m$, and such that
	\[\sum_{i=1}^dx_i(k\sigma)X_i = \sum_{i=1}^dx_i(\sigma)\Ad(k)X_i \hspace{20pt} \forall k\in K.\]
For more details, see \citet{Liao_InvMark} pp.36--37, 83. 

The choice of basis of $\mathfrak{p}$ enables us to view $\Ad(k)$ as a $d\times d$ matrix, for each $k\in K$. A vector $b\in\mathbb{R}^m$ is said to be \emph{$\Ad(K)$-invariant} if
	\[b=\Ad(k)^Tb, \hspace{20pt} \forall k\in K.\]
Similarly, a $d\times d$ real-valued matrix $a=(a_{ij})$ is \emph{$\Ad(K)$-invariant} if  
	\[a=\Ad(k)^Ta\Ad(k) \hspace{20pt} \forall k\in K.\]
A Borel measure $\nu$ on $G$ is called a \emph{L\'evy measure} if $\nu(\{e\}) = 0$, $\nu(U^c)<\infty$, and $\int_G\sum_{i=1}^lx_i(\sigma)^2\nu(d\sigma)$.

We state a useful corollary of the famous Hunt formula. For more details, including a proof, see Section 3.2 of \citet{Liao_InvMark}, pp.~78.

\begin{thm}\label{thm:HuntCor}
Let $\mathcal{A}$ be the infinitesimal generator associated with a $K$-bi-invariant L\'evy process on $G$. Then $C_c^\infty(G)\subseteq\Dom \mathcal{A}$, and there is an $\Ad(K)$-invariant vector $b\in\mathbb{R}^d$, an $\Ad(K)$-invariant, non-negative definite, symmetric $d\times d$ matrix $a:=(a_{ij})$, and a $K$-bi-invariant L\'evy measure $\nu$ such that
	\begin{align*}
	\mathcal{A}f(\sigma) = \sum_{i=1}^d&b_iX_if(\sigma) + \sum_{i,j=1}^da_{ij}X_iX_jf(\sigma) \\
	&+ \int_G\left(f(\sigma\tau)-f(\sigma)-\sum_{i=1}^dx_i(\tau)X_if(\sigma)\right)\nu(d\sigma),
	\end{align*}
for all $f\in C_c^\infty(G)$ and $\sigma\in G$. Moreover, the triple $(b,a,\nu)$ is completely determined by $\mathcal{A}$, and independent of the choice of exponential coordinate functions $x_i,\;i=1,\ldots,d$.

Conversely, given a triple $(b,a,\nu)$ of this kind, there is a unique $K$-bi-invariant convolution semigroup of probability measures on $G$ with infinitesimal generator given by $\mathcal{A}$.
\end{thm}

Since $G$ is semisimple, $\mathfrak{p}$ has no non-zero $\Ad(K)$-invariant elements. This means that for the class of manifold we are considering, $K$-bi-invariant L\'evy generators will take the form 
	\begin{equation}\label{eq:hunt_SSNCT}
	\mathcal{A}f(\sigma) = \sum_{i,j=1}^da_{ij}X_iX_jf(\sigma) + \int_G\left(f(\sigma\tau)-f(\sigma)-\sum_{i=1}^dx_i(\tau)X_if(\sigma)\right)\nu(d\sigma),
	\end{equation}
Given such a L\'evy generator, we write $\mathcal{A}_D = \sum_{i,j=1}^da_{ij}X_iX_j$ for the diffusion part of $\mathcal{A}$. By the discussion surrounding (3.3) in \citet{Liao_InvMark}, pp.~75, $\mathcal{A}_D\in {\bf D}_K(G)$, and so for each $\lambda\in\a*$ there is $\beta(\mathcal{A}_D,\lambda)\in\mathbb{C}$ such that
	\begin{equation}\label{eq:beta(L_D,lambda)}
	\mathcal{A}_D\phi_\lambda = \beta(\mathcal{A}_D,\lambda)\phi_\lambda.
	\end{equation}
Moreover, $\lambda\mapsto\beta(\mathcal{A}_D,\lambda)$ is a $W$-invariant quadratic polynomial function on $\a*$. 

\begin{thm}[Gangolli's L\'evy--Khinchine formula]\label{thm:GLK}
Let $(\mu_t,t\geq 0)$ be a $K$-bi-invariant convolution semigroup of probability measures on $G$ with infinitesimal generator $\mathcal{A}$, and let $\mathcal{A}_D$ denote the diffusion part of $\mathcal{A}$. Then $\hat\mu_t=e^{-t\psi}$, where
	\begin{equation}\label{eq:GangExp}
	\psi(\lambda) = -\beta(\mathcal{A}_D,\lambda) + \int_G(1-\phi_\lambda(\sigma))\nu(d\sigma) \hspace{20pt} \forall\lambda\in\a*,
	\end{equation}
and $\beta(\mathcal{A}_D,\lambda)$ is given by (\ref{eq:beta(L_D,lambda)}).
\end{thm}

This result was first proven in \citet{gangolli}, see also \citet{LiaoWang}. For a proof of the specific statement above, see page 139 of \citet{Liao_InvMark}.

The function $\psi$ given by (\ref{eq:GangExp}) will be called the \emph{Gangolli exponent} of the process $X$. 

\begin{rmk}
If Definition \ref{mydef:convsemigpG} (\ref{probmeas}) is relaxed so that each $\mu_t$ need only satisfy $\mu_t(G)\leq 1$, all of the results described in this subsection continue to hold, except ``sub-'' must be added to some to the terms: convolution semigroups of \emph{sub}-probability measures, \emph{sub-}L\'evy generators, \emph{sub}-diffusion operators, and so on. 
\end{rmk}
%
\subsection{Positive and Negative Definite Functions}\label{subsec:+-fns}
By viewing $\a*$ as a finite-dimensional real vector space, we may consider positive and negative definite functions on $\a*$, defined in the usual way. 

\begin{prop}\label{prop:egs}
\begin{enumerate}
\item For all $\sigma\in G$, $\lambda\mapsto\phi_\lambda(\sigma)$ is positive definite.
\item\label{hatmu} Let $\mu$ be a finite $K$-bi-invariant Borel measure. Then $\hat{\mu}$ is positive definite.
\end{enumerate}
\end{prop}
\begin{proof}
Let $\sigma\in G$, $n\in\mathbb{N}$, $\lambda_1,\ldots,\lambda_n\in\a*$, and $c_1,\ldots,c_n\in\mathbb{C}$, and note that
	\[\sum_{\alpha,\beta=1}^nc_\alpha\overline{c_\beta}e^{(i(\lambda_\alpha-\lambda_\beta)+\rho)A(k\sigma)} = \left|\sum_{\alpha=1}^nc_\alpha e^{(i\lambda_\alpha+\frac{\rho}{2})A(k\sigma)}\right|^2 \geq 0.\]
Therefore, by the Harish-Chandra integral formula (\ref{eq:H-C}),
	\begin{equation}\label{eq:2,7}
	\sum_{\alpha,\beta=1}^nc_\alpha\overline{c_\beta}\phi_{\lambda_\alpha-\lambda_\beta}(\sigma) = \int_K\sum_{\alpha,\beta=1}^nc_\alpha\overline{c_\beta}e^{(i(\lambda_\alpha-\lambda_\beta)+\rho)A(k\sigma)}dk \geq 0.
\end{equation}
Part 1 follows.

For part 2, observe that since (\ref{eq:2,7}) holds for all $c_1,\ldots,c_n$, we can replace each $c_j$ by its complex conjugate. Therefore, $\sum_{\alpha,\beta=1}^n\overline{c_\alpha}c_\beta\phi_{\lambda_\alpha-\lambda_\beta}(\sigma)\geq 0$ for all $\sigma\in G$, $n\in\mathbb{N}$, $\lambda_1,\ldots,\lambda_n\in\a*$, and $c_1,\ldots,c_n\in\mathbb{C}$. Taking complex conjugates,
	\[\sum_{\alpha,\beta=1}^nc_\alpha\overline{c_\beta}\phi_{-(\lambda_\alpha-\lambda_\beta)}(\sigma) = \overline{\sum_{\alpha,\beta=1}^n\overline{c_\alpha}c_\beta\phi_{\lambda_\alpha-\lambda_\beta}}\geq 0,\]
for all $\sigma\in G$, $n\in\mathbb{N}$, $\lambda_1,\ldots,\lambda_n\in\a*$, and $c_1,\ldots,c_n\in\mathbb{C}$, and hence
	\[\sum_{\alpha,\beta=1}^nc_\alpha\overline{c_\beta}\hat{\mu}(\lambda_\alpha-\lambda_\beta) = \int_{\a*}\sum_{\alpha,\beta=1}^nc_\alpha\overline{c_\beta}\phi_{-(\lambda_\alpha-\lambda_\beta)}(\sigma)\mu(d\sigma)\geq 0.\]
\end{proof}

By choosing a basis of $\a*$, we may identify it with $\mathbb{R}^m$, and apply classical results about positive (resp.~negative) definite functions on Euclidean space to functions on $\a*$, to obtain results about positive (resp.~negative) definite functions in this new setting.

One useful application of this is the Schoenberg correspondence, which states that a map $\psi:\a*\to\mathbb{C}$ is negative definite if and only if $\psi(0)\geq 0$ and $e^{-t\psi}$ is positive definite for all $t>0$. This is immediate by the Schoenberg correspondence on $\mathbb{R}^m$ --- see \citet{BergForst} page 41 for a proof.

\begin{prop}
Let $\psi:\a*\to\mathbb{C}$ be the Gangolli exponent of a L\'evy process on $G/K$. Then $\psi$ is negative definite.
\end{prop}
\begin{proof}
Let $X=(X(t),t\geq 0)$ is a L\'evy process on $G/K$, and let $\nu_t$ be the law of $X(t)$, for all $t\geq 0$. Then $(\nu_t,t\geq 0)$ forms a convolution semigroup on $G/K$. By Proposition 1.12 of \citet{Liao_InvMark} (pp.~13), $(\nu_t,t\geq 0)$ arises as the projection onto $G/K$ of a $K$-bi-invariant convolution semigroup $(\mu_t,t\geq 0)$ on $G$. By Proposition \ref{prop:egs}, the spherical transform of each $\mu_t$ is positive definite, and by the Schoenberg correspondence, for each $t\geq 0$, there is a negative definite function $\psi_t$ on $\a*$ such that $\psi_t(0)\geq 0$ and $\hat\mu_t=e^{-\psi_t}$. In fact, since $(\mu_t,t\geq 0)$ is a convolution semigroup, it must be the case that
	\[\hat\mu_t = e^{-t\psi_1}, \hspace{20pt} \forall t\geq 0.\]
By uniqueness of Gangolli exponents, $\psi=\psi_1$, a negative definite function.
\end{proof}

We finish this subsection with a collection of results about negative definite functions, which will be useful in later sections.

\begin{prop}\label{prop:negdef}
Let $\psi:\a*\to\mathbb{C}$ be a continuous negative definite function. Then
\begin{enumerate}
\item\label{sqrt} For all $\lambda,\eta\in\a*$,	
	\[\left|\sqrt{|\psi(\lambda)|}-\sqrt{|\psi(\eta)|}\right|\leq\sqrt{|\psi(\lambda-\eta)|}\]
\item\label{peetre} (Generalised Peetre inequality) For all $s\in\mathbb{R}$ and $\lambda,\eta\in\a*$,
	\[\left(\frac{1+|\psi(\lambda)|}{1+|\psi(\eta)|}\right)^s \leq 2^{|s|}(1+|\psi(\lambda-\eta)|)^{|s|}.\]
\item\label{c_phi} There is a constant $c_\psi>0$ such that
	\begin{equation}\label{eq:c_phi}
	|\psi(\lambda)|\leq c_\psi(1+|\lambda|^2) \hspace{20pt} \forall\lambda\in\a*.
	\end{equation}
\end{enumerate}
\end{prop}
\begin{proof}
These results follow from their analogues on $\mathbb{R}^m$ --- see \citet{hoh} page 16.
\end{proof}

%
\subsection{Spherical Anisotropic Sobolev Spaces}\label{sec:AniSob}
Suppose $\psi$ is a real-valued continuous negative definite function, and let $s\in\mathbb{R}$. We define the \emph{(spherical) anisotropic Sobolev space} associated with $\psi$ and $s$ to be 
	\[H^{\psi,s} := \left\{u\in\mathcal{S}'(K|G|K):\int_G(1+\psi(\lambda))^s|\hat{u}(\lambda)|^2\omega(d\lambda)<\infty\right\},\]
where $\mathcal{S}'(K|G|K)$ denotes the space of $K$-bi-invariant tempered distributions. 
One can check that each $H^{\psi,s}$ is a Hilbert space with respect to the inner product
	\[\langle u,v\rangle_{\psi,s} := \int_{\a*}(1+\psi(\lambda))^s\hat{u}(\lambda)\overline{\hat{v}(\lambda)}\omega(d\lambda), \hspace{20pt} \forall u,v\in H^{\psi,s}.\]
These spaces are a generalisation of the anisotropic Sobolev spaces first introduced by Niels Jacob, see \citet{jacob93}, and developed further by Hoh, see \citet{hoh}. For the special case $\psi(\lambda)=|\rho|^2+|\lambda|^2$, we will write $H^{\psi,s}=H^s$. 
Note also that $H^{\psi,0}=L^2(K|G|K)$, by the Plancherel theorem. In this case, we will omit subscripts and just write $\langle\cdot,\cdot\rangle$ for the $L^2$ inner product.

Note that $\psi$ is a non-negative function, since it is negative definite and real-valued. We impose an additional assumption, namely that there exist constants $r,c>0$ such that
	\begin{equation}\label{eq:psi_est}
	\psi(\lambda)\geq c|\lambda|^{2r} \hspace{20pt} \forall \lambda\in\a*,\;|\lambda|\geq 1.
	\end{equation}
Analogous assumptions are made in \citet{jacob94} (1.5) and \citet{hoh} (4.2), and the role of (\ref{eq:psi_est}) will be very similar.

\begin{thm}\label{thm:aniso}
Let $\psi$ be a real-valued, continuous negative definite symbol, satisfying (\ref{eq:psi_est}). Then
\begin{enumerate}
\item\label{S->H_>S'} $C_c^\infty(K|G|K)$ and $\mathcal{S}(K|G|K)$ are dense in each $H^{\psi,s}$, and we have continuous embeddings
	\[\mathcal{S}(K|G|K)\hookrightarrow H^{\psi,s}\hookrightarrow\mathcal{S}'(K|G|K)\]
\item\label{ctsembed} We have continuous embeddings
	\[H^{\psi,s_2}\hookrightarrow H^{\psi,s_1} \]
whenever $s_1,s_2\in\mathbb{R}$ with $s_2\geq s_1$. In particular, $H^{\psi,s}\hookrightarrow L^2(K|G|K)$ for all $s\geq 0$.
\item Under the standard identification of $L^2(K|G|K)$ with its dual, the dual space of each $H^{\psi,s}$ is isomorphic to $H^{\psi,-s}$, with
	\begin{equation}\label{eq:dual}
	\Vert u\Vert_{\psi,-s} = \sup\left\{\frac{|\langle u,v\rangle|}{\Vert v\Vert_{\psi,s}}:v\in C_c^\infty(K|G|K),\;v\neq 0\right\},
	\end{equation}
for all $s\in\mathbb{R}$.
\item\label{Hs_>Hpsis_>Hrs} For $r>0$ as in equation (\ref{eq:psi_est}), we have continuous embeddings
	\[H^s \hookrightarrow H^{\psi,s} \hookrightarrow H^{rs},\]
for all $s\geq 0$.
\item\label{s3s2s1} Let $s_3>s_2>s_1$. Then for all $\epsilon>0$, there is $c(\epsilon)\geq 0$ such that
	\begin{equation}\label{eq:item}
	\Vert u\Vert_{\psi,s_2} \leq \epsilon\Vert u\Vert_{\psi,s_3} + c(\epsilon)\Vert u\Vert_{\psi,s_1} 
	\end{equation}
for all $u\in H^{\psi,s_3}$.
\item\label{SobEmbed} There exist continuous embeddings
	\[H^{\psi,s}\hookrightarrow C_0(K|G|K)\]
for all $s>\frac{d}{r}$, where $d=\dim(G/K)$.
\end{enumerate}
\end{thm}
For brevity, let
	\begin{equation}\label{eq:<.>}
	\langle\lambda\rangle := \sqrt{1+|\lambda|^2}, \hspace{20pt} \forall\lambda\in\a*,
	\end{equation}
and
	\begin{equation}\label{eq:Psi}
	\Psi(\lambda) := \sqrt{1+\psi(\lambda)}, \hspace{20pt} \forall\lambda\in\a*.
	\end{equation}
The proof of Theorem \ref{thm:aniso} will be given after the next lemma.
\begin{lem}\label{lem:<.>}
Let $M>d=\dim(G/K)$. Then $\langle\cdot\rangle^{-M} \in L^1(\a*,\omega)$.
\end{lem}
\begin{proof}
By standard arguments, one may check that $\int_{\mathbb{R}^d}\langle\xi\rangle^{-M}d\xi<\infty$, for all $M>d$. Writing $p=\frac{\dim N}{2}$, we have $d= \dim\a* + 2p$,
and hence $\int_{\a*}\langle\lambda\rangle^{-M+2p}d\lambda < \infty$ whenever $M>d$. By Proposition 7.2 on page 450 of \citet{helg2}, there are $C_1,C_2>0$ such that
	\begin{equation}\label{eq:hcc-bnd}
	|\hcc(\lambda)|^{-1} \leq C_1 + C_2|\lambda|^p \hspace{20pt} \forall\lambda\in\a*.
	\end{equation}
Let $C>0$ be such that $(C_1 + C_2|\lambda|^p)^2 < C(1+|\lambda|^2)^p$ for all $\lambda\in\a*$. Then
	\begin{equation}\label{eq:Est}
	\int_{\a*}\langle\lambda\rangle^{-M}\omega(d\lambda) = \int_{\a*}\langle\lambda\rangle^{-M}|\hcc(\lambda)|^{-2}d\lambda \leq C\int_{\a*}\langle\lambda\rangle^{-M+2p}d\lambda < \infty,
	\end{equation}
whenever $M>d$.
\end{proof}
\begin{proof}[Proof of Theorem \ref{thm:aniso}]
Much of this theorem may be proved by adapting proofs from the $\mathbb{R}^d$ case. For example, to prove Theorem \ref{thm:aniso} (\ref{S->H_>S'}), let $\mathcal{V}^{\psi,s}$ denote the space of all measurable functions $v$ on $\a*$ for which $\Psi^sv\in L^2(\a*,\omega)^W$, a Hilbert space with respect to the inner product
	\[\langle u,v\rangle = \int_{\al^\ast}\Psi(\lambda)^{2s}u(\lambda)\overline{v(\lambda)}\omega(d\lambda), \hspace{20pt} \forall u,v \in\mathcal{V}^{\psi,s}.\]
By viewing $\a*$ as a real vector space and using inequality (\ref{eq:hcc-bnd}) to relate $\omega$ to Lebesgue measure, the proof of Theorem 3.10.3 on page 208 of \citet{jacobI} may be easily adapted to show that
	\[\mathcal{S}(\a*)^W\hookrightarrow \mathcal{V}^{\psi,s}\hookrightarrow\mathcal{S}'(\a*)^W\]
is continuous. Noting Theorem \ref{thm:ComDiag}, Theorem \ref{thm:aniso} (\ref{S->H_>S'}) follows.

Proofs of Theorem \ref{thm:aniso} (\ref{ctsembed})--(\ref{s3s2s1}) are almost identical to their $\mathbb{R}^d$-based counterparts, see \citet{jacob94} \S 1, or \citet{hoh} pp.~46--48.

By Theorem \ref{thm:aniso} (\ref{Hs_>Hpsis_>Hrs}), Theorem \ref{thm:aniso} (\ref{SobEmbed}) will follow if we can prove the existence of a continuous embeddings
	\begin{equation}\label{eq:↪}
	H^s \hookrightarrow C_0(K|G|K),
	\end{equation}
for all $s>d$. Let $s>d$ and $u\in\mathcal{S}(K|G|K)$. By Lemma \ref{lem:<.>}, $\langle\cdot\rangle^{-s}\in L^2(\al^\ast,\omega)$, and by the spherical inversion formula (\ref{eq:SphInv}),
	\[|u(\sigma)| = \left|\int_{\a*}\phi_\lambda(\sigma)\hat{u}(\lambda)\omega(d\lambda)\right| \leq \int_{\a*}|\hat{u}(\lambda)|\omega(d\lambda) = \int_{\a*}\langle\lambda\rangle^{-s}\langle\lambda\rangle^s|\hat{u}(\lambda)|\omega(d\lambda),\]
for all $\sigma\in G$. By the Cauchy--Schwarz inequality,
	\[|u(\sigma)| \leq \Vert\langle\cdot\rangle^{-s}\Vert_{L^2(\a*,\omega)}\Vert\langle\cdot\rangle^s\hat{u}\Vert_{L^2(\a*,\omega)} = C\Vert u\Vert_s\]
for all $\sigma\in G$, where $C= \Vert\langle\cdot\rangle^{-s}\Vert_{L^2(\a*,\omega)}$. It follows that
	\[\Vert u \Vert_{C_0(K|G|K)} := \sup_{\sigma\in G}|u(\sigma)| \leq C\Vert u\Vert_s.\]
The embedding (\ref{eq:↪}) may then be obtained using a density argument.
\end{proof}

%
\subsection{Pseudodifferential Operators and Their Symbols}\label{subsec:Ops&Symbs}
A measurable mapping $q:G\times\a*\to\mathbb{C}$ will be called a \emph{negative definite symbol} if it is locally bounded, and if for each $\sigma\in G$, $q(\sigma,\cdot)$ is negative definite and continuous. If in addition $q$ is continuous in its first argument, we will call $q$ a \emph{continuous negative definite symbol}. 

Let $\mathcal{M}(G)$ denote the set of all measurable functions on $G$. 
\begin{thm}\label{thm:negdef}
Let $q$ be a negative definite symbol, and for each $f\in C^\infty_c(K|G|K)$ and $\sigma\in G$, define
	\begin{equation}\label{eq:q(sigma,D)}
	q(\sigma, D)f(\sigma) = \int_{\a*}\hat f(\lambda)\phi_\lambda(\sigma)q(\sigma,\lambda)\omega(d\lambda).
	\end{equation}
Then
\begin{enumerate}
\item\label{M} Equation (\ref{eq:q(sigma,D)}) defines a linear operator $q(\sigma,D):C_c^\infty(K|G|K)\to \mathcal{M}(G)$. 
\item\label{C} If $q$ is a continuous negative definite symbol, then $q(\sigma,D):C_c^\infty(K|G|K)\to C(G)$.
\item\label{K} If $q$ is $K$-bi-invariant in its first argument, then $q(\sigma,D)f$ is $K$-bi-invariant for all $f\in C_c^\infty(K|G|K)$.
\end{enumerate}
\end{thm}
\begin{proof}
Theorem \ref{thm:negdef} (\ref{M}) and (\ref{C}) are proved in a similar manner to Theorem 4.5.7 of \citet{jacobI}, while (\ref{K}) is immediate from the $K$-bi-invariance of each spherical function $\phi_\lambda$.
\end{proof}

\begin{mydef}\label{mydef:PSDO}
Operators of the form (\ref{eq:q(sigma,D)}), where $q$ is a negative definite symbol, will be called \emph{(spherical) pseudodifferential operators} on $G$. 
\end{mydef}

An important subclass of these operators first appeared for irreducible symmetric spaces in \cite{App_AusMathSoc}, with the symbol arising as the Gangolli exponent of a $K$-bi-invariant L\'evy process. Note that just as in the classical Euclidean case, the symbols arising from L\'evy processes are spatially independent, in the sense that they are constant in their first argument. We explore some specific examples of this below. In Section \ref{sec:GangOps&HYR}, we introduce a large class of examples pseudodifferential operators with spatial dependence.  

\begin{eg}\label{eg:psdo}
\begin{enumerate}
\item\label{eg:diff} \emph{Diffusion operators with constant coefficients.} Since $G$ is semisimple, the generator of a $K$-bi-invariant diffusion-type L\'evy process $Y$ on $G$ takes the form $\mathcal{A}  \coloneqq  \sum_{i,j=1}^da_{ij}X_iX_j$, where $a=(a_{ij})$ is an $\Ad(K)$-invariant, non-negative definite symmetric $d\times d$ matrix (c.f.~(\ref{eq:hunt_SSNCT})). As already noted, $\mathcal{A}\in{\bf D}_K(G)$; let $\beta(\mathcal{A},\lambda)$ denote the $\phi_\lambda$-eigenvalue of $\mathcal{A}$. Note that $\lambda\mapsto-\beta(\mathcal{A},\lambda)$ is the Gangolli exponent of $Y$ .

We claim that $(\sigma,\lambda)\mapsto-\beta(\mathcal{A},\lambda)$ is a continuous negative definite symbol, and the associated pseudodifferential operator is $-\mathcal{A}$. To see this, let $(\mu_t,t\geq 0)$ denote the convolution semigroup generated by $\mathcal{A}$, and let $(T_t,t\geq 0)$ be the associated Hunt semigroup, as defined in (\ref{eq:Huntsemi}). Then, given $f\in C_c^\infty(K|G|K)$ and $\sigma\in G$,
	\begin{equation}\label{eq:Af}
	\mathcal{A}f(\sigma) = \left.\frac{d}{dt}T_tf(\sigma)\right|_{t=0}.
	\end{equation}
By the spherical inversion formula (\ref{eq:SphInv}), for all $t\geq 0$,
	\[T_tf(\sigma) = \int_G\int_{\a*}\hat f(\lambda)\phi_\lambda(\sigma\tau)\omega(d\lambda)p_t(d\tau).\]
Recalling that $\hat{f}\in\mathcal{S}(\a*)$ whenever $f\in C_c^\infty(K|G|K)$, a Fubini argument may be applied to conclude that
	$T_tf(\sigma) = \int_{\a*}\hat{f}(\lambda)T_t\phi_\lambda(\sigma)\omega(d\lambda)$.
By Proposition \ref{prop:T_tɸ_λ(𝛔)} and Theorem \ref{thm:GLK}, 
	\[T_t\phi_\lambda = \hat\mu_t(\lambda)\phi_\lambda = e^{t\beta(\mathcal{A},\lambda)},\]
and so 
	\[T_tf(\sigma) = \int_{\a*}\hat{f}(\lambda)e^{t\beta(\mathcal{A},\lambda)}\phi_\lambda(\sigma)\omega(d\lambda).\]
By (\ref{eq:Af}), for all $f\in C_c^\infty(K|G|K)$ and $\sigma\in G$,
	\begin{equation}\label{eq:A}
	\mathcal{A}f(\sigma) 
	= \lim_{t\rightarrow 0}\int_{\a*}\hat f(\lambda)\left(\frac{e^{t\beta(\mathcal{A},\lambda)}-1}{t}\right)\phi_\lambda(\sigma)\omega(d\lambda).
	\end{equation}
Now, if $t>0$ and $\lambda\in\a*$, then
	\[\left|\hat f(\lambda)\left(\frac{e^{t\beta(\mathcal{A},\lambda)}-1}{t}\right)\phi_\lambda(\sigma)\right| \leq \left|\hat f(\lambda)\right|\left|\frac{e^{t\beta(\mathcal{A},\lambda)}-1}{t}\right| \leq \left|\hat f(\lambda)\right||\beta(\mathcal{A},\lambda)|.\]
Moreover, $|\hat f||\beta(\mathcal{A},\cdot)|\in L^1(\a*,\omega)^W$, since $\hat f\in\mathcal{S}(\a*)^W$, and $\beta(\mathcal{A},\cdot)$ is a $W$-invariant polynomial function. By the dominated convergence theorem, we may bring the limit through the integral sign in (\ref{eq:A})	to conclude that
	\begin{equation}\label{eq:A''}
	\begin{aligned}
	\mathcal{A}f(\sigma) &= \int_{\a*}\hat f(\lambda)\lim_{t\rightarrow 0}\left(\frac{e^{t\beta(\mathcal{A},\lambda)}-1}{t}\right)\phi_\lambda(\sigma)\omega(d\lambda) \\
	&= \int_{\a*}\hat{f}(\lambda)\phi_\lambda(\sigma)\beta(\mathcal{A},\lambda)\omega(d\lambda)
	\end{aligned}
	\end{equation}
for all $f\in C_c^\infty(K|G|K)$ and $\sigma\in G$. 
\item\label{eg:BM} \emph{Brownian motion.} As a special case of the above, $-\Delta$ is a pseudodifferential operator with symbol $|\rho|^2+|\lambda|^2$.
\item\label{eg:killing} \emph{Killed diffusions.} With minimal effort, the results of Example \ref{eg:psdo} (\ref{eg:diff}) may be extended to include killing. To see this, note first that such operators are always of the form $\mathcal{A}-c$, where $\mathcal{A}$ is a diffusion operator of the form considered above, and $c\geq 0$. The associated $\phi_\lambda$-eigenvalues must satisfy
	\[\beta(\mathcal{A}-c,\lambda) = \beta(\mathcal{A},\lambda)-c,\]
and hence using (\ref{eq:A''}) as well as the spherical inversion theorem,
	\begin{align*}
	(\mathcal{A}-c)f(\sigma) &= \int_{\a*}\hat{f}(\lambda)\phi_\lambda(\sigma)\beta(\mathcal{A},\lambda)\omega(d\lambda) - cf(\sigma) \\
	&= \int_{\a*}\hat{f}(\lambda)\phi_\lambda(\sigma)\beta(\mathcal{A}-c,\lambda)\omega(d\lambda),
	\end{align*}
for all $f\in C_c^\infty(K|G|K)$ and $\sigma\in G$.
\item \emph{L\'evy generators.} More generally, if $\mathcal{A}$ is the infinitesimal generator of a $K$-bi-invariant L\'evy process on $G$, and if $\psi$ is the corresponding Gangolli exponent, then $(\sigma,\lambda)\mapsto\psi(\lambda)$ is a continuous negative definite symbol, and $-\mathcal{A}$ is the corresponding pseudodifferential operator. This is proven in \citet{App_AusMathSoc} Theorem 5.1 in the case where $G/K$ is irreducible, and later in this paper as a special case of Theorem \ref{thm:psdo}.
\end{enumerate}
\end{eg}

%
\section{Gangolli Operators and the Hille--Yosida--Ray Theorem}\label{sec:GangOps&HYR}
We will soon define the class of pseudodifferential operators that will be of primary interest. In this section, we motivate this definition with a short discussion of the Hille--Yosida--Ray theorem, and prove that our class of operators are pseudodifferential operators in the sense of Definition \ref{mydef:PSDO}. We finish the section with some examples.

Let $E$ be a locally compact, Hausdorff space, let $\mathcal{C}$ be a closed subspace of $C_0(E)$, and let $\mathcal{F}(E)$ denote the space of all real-valued functions on $E$. A $C_0$-semigroup $(T_t,t\geq 0)$ defined on $C_0(K|G|K)$ is called \emph{sub-Feller} if for all $f\in C_0(E)$, and all $t\geq 0$,
	\[0\leq f\leq 1 ~\Rightarrow~ 0\leq T_tf\leq 1.\] 
A linear operator $\mathcal{A}:\Dom(\mathcal{A})\to\mathcal{F}(E)$ is said to satisfy the \emph{positive maximum principle}, if for all $f\in\Dom(\mathcal{A})$ and $x_0\in E$ such that $f(x_0)=\sup_{x\in E}f(x)\geq 0$, we have $\mathcal{A}f(x_0)\leq 0$.

The following theorem is an extended version of the Hille--Yosida--Ray theorem, which fully characterises the operators that extend to generators of sub-Feller semigroups on $C_0(E)$. Similar versions in which $E=\mathbb{R}^d$ may found in \citet{hoh}, pp.~53, and \citet{jacobI}, pp.~333. For a proof, see \citet{e&k}, pp.~165.

\begin{thm}[Hille--Yosida--Ray]\label{thm:hyr}
A linear operator $(\mathcal{A},\Dom(\mathcal{A}))$ on $C_0(E)$ is closable and its closure generates a strongly continuous, sub-Feller semigroup on $C_0(E)$ if and only if the following is satisfied:
\begin{enumerate}
\item\label{denselydefined} $\Dom(\mathcal{A})$ is dense in $C_0(E)$,
\item\label{PMP} $\mathcal{A}$ satisfies the positive maximum principle, and
\item\label{dense} There exists $\alpha>0$ such that $\Ran(\alpha I-\mathcal{A})$ is dense in $C_0(E)$.
\end{enumerate}
\end{thm}

In their papers \citet{AppNgan_PMPLie,AppNgan_PMPSS}, Applebaum and Ngan found necessary and sufficient conditions for an operator defined on $C_c^\infty(K|G|K)$ to satisfy Theorem \ref{thm:hyr} (\ref{PMP}), for the cases $E=G$, $G/K$ and $K|G|K$. We will focus primarily on the case $E=K|G|K$, since this is the realm in which the spherical transform is available. 

A mapping $\nu:G\times\mathcal{B}\to[0,\infty]$ will be called a \emph{$K$-bi-invariant L\'evy kernel} if it is $K$-bi-invariant in its first argument, and if for all $\sigma\in G$, $\nu(\sigma,\cdot)$ is a $K$-bi-invariant L\'evy measure. Fix a system of exponential coordinate functions, as defined in Definition \ref{mydef:expcoords}, and adopt all of the notation conventions from this definition. 

\begin{mydef}\label{mydef:GangolliOperator}
An operator $\mathcal{A}:C_c^\infty(K|G|K)\to \mathcal{F}(G)$ will be called a \emph{Gangolli operator} if there exist mappings $c, a_{i,j}\in\mathcal{F}(K|G|K)$ ($1\leq i,j\leq d$), as well as a $K$-bi-invariant L\'evy kernel $\nu$, such that for all $f\in C_c^\infty(K|G|K)$ and $\sigma\in G$,
	\begin{equation}\label{eq:GangOp}
	\begin{aligned}
	\mathcal{A}f(\sigma) = -c(\sigma)&f(\sigma) + \sum_{i,j=1}^da_{i,j}(\sigma)X_iX_jf(\sigma) \\
	&+ \int_G\left(f(\sigma\tau)-f(\sigma)-\sum_{i=1}^dx_i(\tau)X_if(\sigma)\right)\nu(\sigma,d\tau),
	\end{aligned}
	\end{equation}	
and if for all $\sigma\in G$,
\begin{enumerate}
\item $c(\sigma)\geq 0$.
\item $a(\sigma):=(a_{i,j}(\sigma))$ is an $\Ad(K)$-invariant, non-negative definite, symmetric matrix.
\end{enumerate}
\end{mydef}

\begin{rmks}
\begin{enumerate}
\item Gangolli operators were first introduced in \citet{AppNgan_PMPSS} in compact symmetric spaces and for a more restrictive form of (\ref{eq:GangOp}).
By Theorem 3.2 (3) of \citet{AppNgan_PMPSS}, Gangolli operators map into $\mathcal{F}(K|G|K)$, and satisfy the positive maximum principle.
\item Equation (\ref{eq:GangOp}) may be viewed as a spatially dependent generalisation of (\ref{eq:hunt_SSNCT}), with an additional killing term $c$. As with previously, the absence of a drift term is due to the semisimplicity of $G$. 
\end{enumerate}
\end{rmks}

For a Gangolli operator $\mathcal{A}$ given by (\ref{eq:GangOp}), and for each $\sigma\in G$, we will denote by $\mathcal{A}^\sigma$ the operator obtained by freezing the coefficients of $\mathcal{A}$ at $\sigma$. Explicitly, for all $f\in C_c^\infty(K|G|K)$ and $\sigma'\in G$,
	\begin{align*}
	\mathcal{A}^\sigma f(\sigma') = -c(\sigma)&f(\sigma') + \sum_{i,j=1}^da_{i,j}(\sigma)X_iX_jf(\sigma') \\
	&+ \int_G\left(f(\sigma'\tau)-f(\sigma')-\sum_{i=1}^dx_i(\tau)X_if(\sigma')\right)\nu(\sigma,d\tau).
	\end{align*}	
For each $\sigma\in G$, $\mathcal{A}^\sigma$ is the generator of a killed $K$-bi-invariant L\'evy process on $G$. We continue to adopt the notation $\mathcal{A}^\sigma_D$ for the diffusion part, and $\beta(\mathcal{A}_D^\sigma,\lambda)$ for the $\phi_\lambda$-eigenvalue of $\mathcal{A}_D^\sigma$. 
	
Consider the following continuity conditions on the coefficients $(b,a,\nu)$ of $\mathcal{A}$:
\begin{itemize}
\item[(c1)] $c,a_{ij}$ are continuous, for $1\leq i,j\leq d$.
\item[(c2)] For each $f\in C_b(K|G|K)$, the mappings $\sigma\mapsto\int_Uf(\tau)\sum_{i=1}^dx_i(\tau)^2\nu(\sigma,d\tau)$ and $\sigma\mapsto\int_{U^c}f(\tau)\nu(\sigma,d\tau)$ are continuous from $G$ to $[0,\infty)$.
\end{itemize}

\begin{lem}\label{lem:GangSymb}
Let $\mathcal{A}$ be a Gangolli operator, and define $q:G\times\a*\to\mathbb{C}$ by
	\begin{equation}\label{eq:GangSymb}
	q(\sigma,\lambda) = -\beta(\mathcal{A}_D^\sigma,\lambda) + \int_G(1-\phi_\lambda(\tau))\nu(\sigma,d\tau), \hspace{20pt} \forall\sigma\in G,\lambda\in\a*.
	\end{equation}
Suppose (c1) and (c2) hold. Then $q$ is a continuous negative definite symbol.
\end{lem}
\begin{proof}
That $q$ is continuous in its first argument is immediate from (c1) and (c2). Fix $\sigma\in G$ and consider $q(\sigma,\cdot)-c(\sigma)$. By Theorem \ref{thm:HuntCor}, there is a convolution semigroup $(\mu^\sigma_t,t\geq 0)$ generated by $\mathcal{A}^\sigma+c(\sigma)$, and by Theorem \ref{thm:GLK}, the corresponding Gangolli exponent is a continuous negative definite mapping on $\a*$, given by
	\[\psi^\sigma(\lambda) = q(\sigma,\lambda)-c(\sigma) \hspace{20pt} \forall\lambda\in\a*.\]
Therefore $q(\sigma,\cdot)$ is continuous, and negative definite since for fixed $\sigma$, $c(\sigma)$ is a non-negative constant.
\end{proof}

\begin{mydef}
The symbols described by Lemma \ref{lem:GangSymb} will be referred to as \emph{Gangolli symbols}, due to their connection with Gangolli's L\'evy--Khinchine formula.
\end{mydef}

\begin{rmks} 
\begin{enumerate}
\item Gangolli exponents are precisely those Gangolli symbols constant in their first argument.
\item The set of all Gangolli symbols forms a convex cone. 
\end{enumerate}
\end{rmks}

\begin{thm}\label{thm:psdo}
Let $\mathcal{A}$ and $q$ be as in Lemma \ref{lem:GangSymb}. Then $-\mathcal{A}$ is a pseudodifferential operator with symbol $q$.
\end{thm}
\begin{proof}
By Theorem \ref{thm:negdef} and Lemma \ref{lem:GangSymb}, $f\mapsto-\int_{\a*}\hat f(\lambda)\phi_\lambda(\sigma)q(\sigma,\lambda)\omega(d\lambda)$ is a well-defined mapping from $C_c^\infty(K|G|K)\to C(G)$. We show that it is equal to $\mathcal{A}$.

Let $\mathcal{A}_J$ denote the non-local (i.e.~jump) part of $\mathcal{A}$, so that
	\begin{equation}\label{eq:A_Jf}
	\mathcal{A}_Jf(\sigma) = \int_G\left(f(\sigma\tau)-f(\sigma)-\sum_{i=1}^dx_i(\tau)X_if(\sigma)\right)\nu(\sigma,d\tau)
	\end{equation}
for all $f\in C_c^\infty(K|G|K)$ and $\sigma\in G$. By design,
	\begin{equation}\label{eq:A=A^sigma_D+A_J}
	\mathcal{A}f(\sigma) = \mathcal{A}_D^\sigma f(\sigma) + \mathcal{A}_Jf(\sigma), \hspace{20pt} \forall f\in C_c^\infty(K|G|K),\;\sigma\in G.
	\end{equation}

For the diffusion part of $\mathcal{A}$, note that for each $\sigma\in G$, $\mathcal{A}^\sigma_D$ is an operator of the form considered in Example \ref{eg:psdo} (\ref{eg:killing}), and in particular satisfies
	\begin{equation}\label{eq:A^sigma_D}
	\mathcal{A}_D^\sigma f(\sigma) =  \int_{\a*}\hat{f}(\lambda)\beta(\mathcal{A}_D^\sigma,\lambda)\phi_\lambda(\sigma)\omega(d\lambda),
	\end{equation}
for all $f\in C_c^\infty(K|G|K)$.

Consider now the jump part $\mathcal{A}_J$. By Lemma 2.3 on page 39 of \citet{Liao_InvMark}, for each fixed $\sigma\in G$, and for all $f\in C_b^2(K|G|K)$, the integrand on the right-hand side of (\ref{eq:A_Jf}) is absolutely integrable with respect to $\nu(\sigma,\cdot)$. Therefore, (\ref{eq:A_Jf}) may be used to extend the domain of $\mathcal{A}_J$ so as to include $C_b^2(K|G|K)$. We do so now, and (without any loss of precision) denote the extension by $\mathcal{A}_J$. 

Let us proceed similarly to \citet{AppNgan_PMPSS} Section 5, and define for each $\sigma\in G$ a linear functional $\mathcal{A}_{J,\sigma}:C_b^2(K|G|K)\to\mathbb{C}$ by
	\[\mathcal{A}_{J,\sigma} f:= \mathcal{A}_J\left(L_\sigma^{-1}f\right)(\sigma), \hspace{20pt} \forall\sigma\in G,\; f\in C_b^2(K|G|K).\]
Then $\mathcal{A}_Jf(\sigma) = \mathcal{A}_{J,\sigma}(L_\sigma f)$, and hence
	\[\mathcal{A}_{J,\sigma}\phi_\lambda = \int_G\left(L_\sigma^{-1}\phi_\lambda(\sigma\tau) - L_\sigma^{-1}\phi_\lambda(\sigma) - \sum_{i=1}^dx_i(\tau)X_iL_\sigma^{-1}\phi_\lambda(\sigma)\right)\nu(\sigma,d\tau),\]
for all $\sigma\in G$ and $f\in C_b^2(K|G|K)$. Moreover, the integrand on the right-hand side is absolutely $\nu(\sigma,\cdot)$-integrable, for all $\lambda\in\a*$ and $\sigma\in G$. Since $\phi_\lambda(e)=1$, and $X\phi_\lambda(e)=0$ for all $X\in\mathfrak{p}$ (Theorem 5.3 (b) of \citet{Liao_InvMark}), 
	\begin{align*}
	L_\sigma^{-1}\phi_\lambda(\sigma\tau) - L_\sigma^{-1}\phi_\lambda(\sigma) - \sum_{i=1}^dx_i(\tau)X_iL_\sigma^{-1}\phi_\lambda(\sigma) = \phi_\lambda(\tau) - 1.
	\end{align*}
Thus, for all $\lambda\in\a*$ and $\sigma\in G$, $\phi_\lambda-1$ is absolutely $\nu(\sigma,\cdot)$-integrable, and
	\begin{equation}\label{eq:P^sigma}
	\mathcal{A}_{J,\sigma}\phi_\lambda = \int_G\left(\phi_\lambda(\tau) - 1\right)\nu(\sigma,d\tau).
		\end{equation}
A standard argument involving the functional equation (\ref{eq:SphIntForm}) for spherical functions may now be applied in precisely the same way as in \citet{AppNgan_PMPSS} (5.3)--(5.7), to infer that 
	\begin{equation}\label{eq:P}
	\mathcal{A}_J\phi_\lambda(\sigma) = \int_G(\phi_\lambda(\sigma\tau) - \phi_\lambda(\sigma))\nu(\sigma,d\tau) = \int_G(\phi_\lambda(\tau) - 1)\phi_\lambda(\sigma)\nu(\sigma,d\tau)
	\end{equation}
for all $\sigma\in G$ and $\lambda\in\a*$.

Finally, let $f\in C_c^\infty(K|G|K)$, and observe that by the spherical inversion formula 
	\begin{equation}\label{eq:Pf(sigma)}
	\begin{aligned}
	\mathcal{A}_Jf(\sigma) &= \int_G\Bigg(\int_{\a*}\phi_\lambda(\sigma\tau)\hat{f}(\lambda)\omega(d\lambda) - \int_{\a*}\phi_\lambda(\sigma)\hat{f}(\lambda)\omega(d\lambda) \\
	&\hspace{70pt}-\sum_{i=1}^dx_i(\tau)X_i\left[\int_{\a*}\phi_\lambda\hat{f}(\lambda)\omega(d\lambda)\right](\sigma)\Bigg)\nu(\sigma,d\tau)
	\end{aligned}
	\end{equation}
	
\begin{claim}
For all $X\in\mathfrak{p}$ and $f\in C_c^\infty(K|G|K)$,
	\[X\left[\int_{\a*}\phi_\lambda\hat{f}(\lambda)\omega(d\lambda)\right](\sigma)=\int_{\a*}X\phi_\lambda(\sigma)\hat{f}(\lambda)\omega(d\lambda).\]
\end{claim}
\begin{pf}
This is a fairly standard differentiation-through-integration-sign argument. First note that by translation invariance of $X$, it suffices to prove the claim for $\sigma=e$. Now,
	\begin{align*}
	X\left[\int_{\a*}\phi_\lambda\hat{f}(\lambda)\omega(d\lambda)\right](e) &= \left.\frac{d}{dt}\int_{\a*}\phi_\lambda(\exp tX)\hat{f}(\lambda)\omega(d\lambda)\right|_{t=0} \\
	&= \lim_{t\rightarrow 0}\int_{\a*}\frac{\phi_\lambda(\exp tX)-1}{t}\hat{f}(\lambda)\omega(d\lambda).
	\end{align*}
The claim will follow if we can apply the dominated convergence theorem to bring the above limit through the integral sign. By the mean value theorem, for each $t>0$ and $\lambda\in\a*$, 
	\[\frac{\phi_\lambda(\exp tX)-1}{t}=X\phi_\lambda(\exp t'X),\]
for some $0<t'<t$, and hence $\left|\frac{\phi_\lambda(\exp tX)-1}{t}\right| \leq \Vert X\phi_\lambda\Vert_\infty$  for all $t>0$. By \citet{helg2SuppNotes} Theorem 1.1 (iii), $\Vert X\phi_\lambda\Vert_\infty\leq C(1+|\lambda|)$, for some some constant $C>0$. Thus, for $f\in C_c^\infty(K|G|K)$, $\lambda\in\a*$ and $t>0$,
	\[\left|\frac{\phi_\lambda(\exp tX)-1}{t}\hat{f}(\lambda)\right| \leq C(1+|\lambda|)|\hat{f}(\lambda)|,\]
and clearly $C(1+|\cdot|)\hat{f} \in L^1(\a*)^W$, since $\hat{f}\in\mathcal{S}(\a*)$. Hence we may apply dominated convergence as desired, and the claim follows.
\end{pf}

Applying the claim to (\ref{eq:Pf(sigma)}), for $f\in C_c^\infty(K|G|K)$ and $\sigma\in G$,
	\begin{align*}
	\mathcal{A}_Jf(\sigma) &= \int_G\Bigg(\int_{\a*}\phi_\lambda(\sigma\tau)\hat{f}(\lambda)\omega(d\lambda) - \int_{\a*}\phi_\lambda(\sigma)\hat{f}(\lambda)\omega(d\lambda) \\
	&\hspace{120pt}-\sum_{i=1}^dx_i(\tau)\int_{\a*}X_i\phi_\lambda(\sigma)\hat{f}(\lambda)\omega(d\lambda)\Bigg)\nu(\sigma,d\tau) \\
	&=\int_G\int_{\a*}\hat f(\lambda)\left(\phi_\lambda(\sigma\tau)-\phi_\lambda(\sigma)-\sum_{i=1}^dx_i(\tau)X_i\phi_\lambda(\sigma)\right)\omega(d\lambda)\nu(\sigma,d\tau). 
	\end{align*}
By the Fubini theorem,
	\begin{align*}
	\mathcal{A}_Jf(\sigma) &= \int_{\a*}\hat f(\lambda)\int_G\left(\phi_\lambda(\sigma\tau)-\phi_\lambda(\sigma)-\sum_{i=1}^dx_i(\tau)X_i\phi_\lambda(\sigma)\right)\nu(\sigma,d\tau)\omega(d\lambda) \\
	&= \int_{\a*}\hat f(\lambda)\mathcal{A}_J\phi_\lambda(\sigma)\omega(d\lambda) 
	\end{align*}
for all $f\in C_c^\infty(K|G|K)$ and $\sigma\in G$. It follows by (\ref{eq:P}) that
	\begin{equation}\label{eq:P'}
	\mathcal{A}_Jf(\sigma) = \int_{\a*}\hat f(\lambda)\phi_\lambda(\sigma)\int_G(\phi_\lambda(\tau) - 1)\nu(\sigma,d\tau)\omega(d\lambda)
	\end{equation}
for all $f\in C_c^\infty(K|G|K)$ and $\sigma\in G$. 

The result now follows by substituting (\ref{eq:P'}) and (\ref{eq:A^sigma_D}) into (\ref{eq:A=A^sigma_D+A_J}).
\end{proof}

\begin{eg}\label{eg:GangSymbs}
\begin{enumerate}
\item Let $u\in C(K|G|K)$ be non-negative, and let $v:\a*\to\mathbb{C}$ be a Gangolli exponent. Then $q:G\times\a*\to\mathbb{C}$ given by
	\[q(\sigma,\lambda) = u(\sigma)v(\lambda) \hspace{20pt} \forall\sigma\in G,\;\lambda\in\a*\]
is a Gangolli symbol. Indeed, by Theorem \ref{thm:GLK}, there exists a sub-diffusion operator $\mathcal{L}\in{\bf D}_K(G)$ and a $K$-bi-invariant L\'evy measure $\nu$ such that for all $\lambda\in\a*$,
	\[v(\lambda) = -\beta(\mathcal{L},\lambda) + \int_G(1-\phi_\lambda(\sigma))\nu(d\tau),\]
and hence for all $\sigma\in G$ and $\lambda\in\a*$,
	\[q(\sigma,\lambda) = -\beta(u(\sigma)\mathcal{L},\lambda) + \int_G(1-\phi_\lambda(\sigma))u(\sigma)\nu(d\tau).\]
If $\mathcal{L} = -c + \sum_{i,j=1}^da_{ij}X_iX_j$, where $c\geq 0$ and $a=(a_{ij})$ is an $\Ad(K)$-invariant, non-negative definite symmetric matrix, then the characteristics are $q$ are 
	\[c(\sigma) := u(\sigma)c, \hspace{10pt} a(\sigma) = u(\sigma)a,\hspace{5pt} \text{and} \hspace{5pt} \nu(\sigma,\cdot)=u(\sigma)\nu.\] 
Since $u$ is  non-negative, continuous and $K$-bi-invariant, the conditions of Definition \ref{mydef:GangolliOperator} are easily verified for these characteristics, as are (c1) and (c2).
\item \emph{Hyperbolic plane.} As described in \citet{helg2} (pp.~29--31), the Poincar\'e disc model $D$ of the hyperbolic plane is isomorphic to $SU(1,1)/SO(2)$. Moreover, $D$ is a symmetric space of noncompact type, with spherical functions are given by the Legendre functions
	\[\phi_\lambda(z) = P_{\frac{1}{2}+i\lambda}\big(\cosh d_{\mathbb{H}}(0,z)\big), \hspace{20pt} \forall z\in D, \lambda\in\mathbb{R}.\]
(see \citet{helg1} Proposition 2.9, pp.~406). Since $D$ is irreducible and $\dim D>1$, by Theorem 3.3 of \citet{AppNgan_PMPSS}, diffusion operators on $D$ must be multiples of the Laplace--Beltrami operator, and the symbols of Feller processes take the simplified form
	\[q(z,\lambda)=c(z)\left(\frac{1}{4}+\lambda^2\right)+\int_0^\infty\big\{1-P_{\frac{1}{2}+i\lambda}(\cosh r)\big\}\nu(z,dr),\]
for all $z\in D$ and $\lambda\in\mathbb{R}$. The constant coefficient (i.e.~L\'evy) case of this formula was discovered by Getoor --- see \citet{getoor} Theorem 7.4.
\end{enumerate}
\end{eg}
%
\section{Construction of Sub-Feller Semigroups}\label{sec:J&H}
In this section we tackle the third condition of Hille--Yosida--Ray (Theorem \ref{thm:hyr}), when $E=K|G|K$. To this end, we seek conditions on a symbol $q$ so that, for some $\alpha>0$,
	\begin{equation}\label{eq:hyr3}
	\overline{\Ran(\alpha+q(\sigma,D))}=C_0(K|G|K).
	\end{equation}
Our approach is based primarily on \citet{jacob94} and \citet{hoh} Section 4. Now that we are on the level of operators, there are more arguments that closely resemble these sources. In these cases, proofs are not expanded in great detail, and may be omitted entirely to save space. Instead, we aim to emphasise what does not carry over from the Euclidean space setting. 

For a mapping $q:G\times\a*\to\mathbb{R}$ and for each $\lambda,\eta\in\a*$, $\sigma\in G$, define
	\begin{equation}\label{eq:F}
	F_{\lambda,\eta}(\sigma) = \phi_{-\lambda}(\sigma)q(\sigma,\eta).
	\end{equation}
Observe that if $q(\cdot,\eta)\in L^2(K|G|K)$ for all $\eta\in\a*$, then $F_{\lambda,\eta}\in L^2(K|G|K)$, and we may consider the spherical transform $\hat{F}_{\lambda,\eta}\in L^2(\a*,\omega)$, given by
	\[\hat{F}_{\lambda,\eta}(\mu) = \int_G\phi_{-\mu}(\sigma)\phi_{-\lambda}(\sigma)q(\sigma,\eta)d\sigma, \hspace{20pt} \forall\mu\in\a*.\] 

To motivate the introduction of $F_{\lambda,\eta}$, consider the case $G=\mathbb{R}^d$, $K=\{0\}$. In this case, the so-called frequency shift property for the Fourier transform says that
	\begin{equation}\label{eq:hat{q}}
	\begin{aligned}
	\hat{F}_{\lambda,\eta}(\mu) &= \frac{1}{(2\pi)^{d/2}}\int_{\mathbb{R}^d}e^{-i\mu\cdot x}e^{-i\lambda\cdot x}q(x,\eta)dx \\
	&= \frac{1}{(2\pi)^{d/2}}\int_{\mathbb{R}^d}e^{-i(\mu+\lambda)\cdot x}q(x,\eta)dx = \hat{q}(\lambda+\mu,\eta),
	\end{aligned}
	\end{equation}
where $^{\wedge}$ denotes the Fourier transform taken in the first argument of $q$. \citet{hoh} and \citet{jacobII} make use of bounds on $\hat{q}(\lambda-\mu,\eta)$, and $\hat{F}_{\lambda,\eta}(-\mu)$ will assume an analogous role in work to come.

As in previous work, let $\psi:\a*\to\mathbb{R}$ be a fixed real-valued, continuous negative definite function satisfying (\ref{eq:psi_est}) for some fixed $r>0$. The next lemma is an analogue of Lemma 2.1 of \citet{jacob94}. See also \citet{hoh} Lemma 4.2, pp.~48. The primary difference in this work is the presence of integer powers of $\sqrt{-\Delta}$, which replace the multinomial powers of $\frac{\partial}{\partial x_1},\ldots,\frac{\partial}{\partial x_d}$ of the $\mathbb{R}^d$ setting. 

One advantage of this approach is that $(-\Delta)^{\beta/2}$ ($\beta\in\mathbb{N}$) has a global definition that does not depend on our choice of local coordinates. Another advantage is that we know its symbol --- see equations (\ref{eq:phi_mu}) and (\ref{eq:psdo}) below.

\begin{lem}\label{lem:Φ_β}
Let $M\in\mathbb{N}$, $q:G\times\a*\to\mathbb{R}$ and suppose $q(\cdot,\lambda)\in C^M_c(K|G|K)$ for all $\lambda\in\a*$. Suppose that for each $\beta\in\{0,1,\ldots,M\}$, there is a non-negative function $\Phi_\beta\in L^1(K|G|K)$ such that
	\begin{equation}\label{eq:Phi_betacond}
		\left|(-\Delta)^{\beta/2}F_{\lambda,\eta}(\sigma)\right| \leq \Phi_\beta(\sigma)\langle\lambda\rangle^M(1+\psi(\eta)),
		\end{equation}
for all $\lambda,\eta\in\a*$, $\sigma\in G$. Then there is a constant $C_M>0$ such that
	\begin{equation}\label{eq:Phi_beta}
	\left|\hat{F}_{\lambda,\eta}(\mu)\right| \leq C_M\sum_{\beta=0}^M\Vert\Phi_\beta\Vert_1\langle\lambda+\mu\rangle^{-M}(1+\psi(\eta)),
	\end{equation}
for all $\lambda,\mu,\eta\in\a*$, where $\Vert\cdot\Vert_1$ denotes the usual norm on the Banach space $L^1(K|G|K)$.
\end{lem}
\begin{rmks}\label{rmks:Φ_βlem}
\begin{enumerate}
\item As in (\ref{eq:<.>}), $\langle\lambda\rangle:=\sqrt{1+|\lambda|^2}$.
\item The condition (\ref{eq:Phi_betacond}) may seem quite obscure. The role of $\langle\lambda+\mu\rangle$ will hopefully become apparent in the proof of Theorem \ref{thm:H^psi,2}. For examples where it is satisfied, see \S\ref{sec:EGs}.
\item\label{calc} Under the conditions of the lemma, and using the Fubini theorem, we have the following: for all $u\in C_c^\infty(K|G|K)$ and $\lambda\in\a*$,
	\begin{equation}\label{eq:calc}
	\begin{aligned}
	(q(\sigma,D)u)^{\wedge}(\lambda) &= \int_G\int_{\a*}\phi_{-\lambda}(\sigma)\phi_\eta(\sigma)q(\sigma,\eta)\hat{u}(\eta)\omega(d\eta)d\sigma \\
	&= \int_{\a*}\left(\int_G\phi_\eta(\sigma)F_{\lambda,\eta}(\sigma)d\sigma\right)\hat{u}(\eta)\omega(d\eta) \\
	&= \int_{\a*}\hat{F}_{\lambda,\eta}(-\eta)\hat{u}(\eta)\omega(d\eta).
	\end{aligned}
	\end{equation}

Fubini's theorem does indeed apply here --- a suitable bound for the integrand on the first line of (\ref{eq:calc}) may be found by noting that, by (\ref{eq:Phi_betacond}),
	\begin{equation}\label{eq:anotherFubini}
	|\phi_{-\lambda}(\sigma)\phi_\eta(\sigma)q(\sigma,\eta)\hat{u}(\eta)|\leq |q(\sigma,\eta)||\hat{u}(\eta)|\leq \Phi_0(\sigma)\big(1+\psi(\eta)\big)|\hat{u}(\eta)|,
	\end{equation}
for all $\lambda,\eta\in\a*$ and $\sigma\in G$. By Theorem \ref{thm:ComDiag}, $\hat{u}\in\mathcal{S}(\a*)$, and the usual bound (\ref{eq:hcc-bnd}) on the density of Plancherel measure may be applied, similarly to (\ref{eq:Est}), to conclude that the right-hand side of (\ref{eq:anotherFubini}) is $\omega(d\eta)\times d\sigma$-integrable. 
\end{enumerate}
\end{rmks}

\begin{proof}[Proof of Lemma \ref{lem:Φ_β}]
Let $\beta\in\{0,1,\ldots,M\}$ and $\lambda,\eta\in\a*$ be fixed. The  fractional Laplacian $(-\Delta)^{\beta/2}$ satisfies a well-known eigenrelation
	\begin{equation}\label{eq:phi_mu}
	(-\Delta)^{\beta/2}\phi_\mu = \left(|\rho|^2+|\mu|^2\right)^{\beta/2}\phi_\mu, \hspace{20pt} \forall\mu\in\a*,
	\end{equation}
which may be proven using subordination methods and properties of the Laplace-Beltrami operator on a symmetric space, using similar techniques to Section 5.7 of \citet{AppLieProb}, pp.~154--7. One can also show using standard methods that
	\begin{equation}\label{eq:psdo}
	\left((-\Delta)^{\beta/2}f\right)^{\wedge}(\mu) = \left(|\rho|^2+|\mu|^2\right)^{\beta/2}\hat{f}(\mu),
	\end{equation}
for all $f\in C^M_c(K|G|K)$ and $\mu\in\a*$. Then, using the definition of  the spherical transform,
	\[(|\rho|^2+|\mu|^2)^{\beta/2}\hat{f}(\mu) = \int_G\phi_{-\mu}(\sigma)(-\Delta)^{\beta/2}f(\sigma)d\sigma, \]
for all $f\in C_c^M(K|G|K)$ and all $\mu\in\a*$. Applying this to $f=F_{\lambda,\eta}$, we have for all $\mu\in\al^\ast$,	
	\begin{align*}
	\left|\left(|\rho|^2+|\mu|^2\right)^{\beta/2}\hat{F}_{\lambda,\eta}(\mu)\right| &\leq \int_G|\phi_{-\mu}(\sigma)|\left|(-\Delta)^{\beta/2}F_{\lambda,\eta}(\sigma)\right|d\sigma \\
	&\leq \int_G\Phi_\beta(\sigma)\langle\lambda\rangle^M(1+\psi(\eta))d\sigma = \Vert\Phi_\beta\Vert_1\langle\lambda\rangle^M(1+\psi(\eta)),
	\end{align*}
and summing over $\beta$,
	\begin{equation}\label{eq:sumoverbeta}
	\sum_{\beta=0}^M\left(|\rho|^2+|\mu|^2\right)^{\beta/2}\left|\hat{F}_{\lambda,\eta}(\mu)\right| \leq \sum_{\beta=0}^M\Vert\Phi_\beta\Vert_1\langle\lambda\rangle^M(1+\psi(\eta)),
	\end{equation}
for all $\lambda,\mu,\eta\in\a*$. Let $C'_M>0$ be the smallest positive number such that
	\[\langle\mu\rangle^M \leq C'_M\sum_{\beta=0}^M\left(|\rho|^2+|\mu|^2\right)^{\beta/2} \hspace{20pt} \forall\mu\in\a*.\]
Then, rearranging (\ref{eq:sumoverbeta}),
	\begin{equation}\label{eq:lem}
	\left|\hat{F}_{\lambda,\eta}(\mu)\right| \leq C'_M\sum_{\beta=0}^M\Vert\Phi_\beta\Vert_1\langle\mu\rangle^{-M}\langle\lambda\rangle^M(1+\psi(\eta)),
	\end{equation}
for all $\lambda,\mu,\eta\in\a*$.

Finally, observe that by Peetre's inequality (see Proposition \ref{prop:negdef} (\ref{peetre})),
	\[\langle\lambda\rangle^M\langle\lambda+\mu\rangle^{-M} = \left(\frac{1+|\lambda|^2}{1+|\lambda+\mu|^2}\right)^{M/2}  \leq 2^{M/2}(1+|\mu|^2)^{M/2} = 2^{M/2}\langle\mu\rangle^M\]
for all $\lambda,\mu\in\a*$. Therefore, for all $\lambda,\mu\in\a*$,
	\[\langle\mu\rangle^{-M}\langle\lambda\rangle^M\leq 2^{M/2}\langle\lambda+\mu\rangle^{-M}\]
and by (\ref{eq:lem}), 
	\[\left|\hat{F}_{\lambda,\eta}(\mu)\right| \leq 2^{M/2}C'_M\sum_{\beta=0}^M\Vert\Phi_\beta\Vert_1\langle\lambda+\mu\rangle^{-M}(1+\psi(\eta))\]
The result now follows by taking $C_M=2^{M/2}C'_M$.
\end{proof}

\begin{rmk}\label{rmk:C_M}
The constant 
	\begin{equation}\label{eq:C_M}
	C_M:=2^{M/2}\sup_{\lambda\in\a*}\frac{\langle\lambda\rangle^M}{\sum_{\beta=0}^M\big(|\rho|^2+|\lambda|^2\big)^{\beta/2}}
	\end{equation}
appearing in the proof of Lemma \ref{lem:Φ_β} will remain relevant throughout this chapter.
\end{rmk}
Let now $q:G\times\a*\to\mathbb{R}$ be a continuous negative definite symbol, $K$-bi-invariant in its first argument, and $W$-invariant in its second (for example, $q$ could be taken to be a Gangolli symbol, as in (\ref{eq:GangSymb})). Similarly to \citet{jacob94} \S 4 and \citet{hoh} (4.26), we write
	\begin{equation}\label{eq:q}
	q(\sigma,\lambda) = q_1(\lambda)+q_2(\sigma,\lambda), \hspace{20pt} \forall\sigma\in G,\lambda\in\a*,
	\end{equation}
where $q_1(\lambda)=q(\sigma_0,\lambda)$ and $q_2(\sigma,\lambda)=q(\sigma,\lambda)-q(\sigma_0,\lambda)$, for some fixed $\sigma_0\in G$. Observe that $q_1$ is necessarily a negative definite symbol. Though $q_2$ may not be, we may still define the operator $q_2(\sigma,D)$ in a meaningful way, by  
	\[q_2(\sigma,D) := q(\sigma,D) - q_1(D) = \int_{\a*}\phi_\lambda(\sigma)q_2(\sigma,\lambda)\hat{f}(\lambda)\omega(d\lambda), \hspace{20pt} \forall \sigma\in G.\]
By decomposing $q$ in this way, we view it as a perturbation of a negative definite function $q_1$ by $q_2$. The assumptions we place on $q$ will control the size of this perturbation, as well as ensuring certain regularity properties of $q(\sigma,D)$ acting on the anisotropic Sobolev spaces introduced in Section \ref{sec:AniSob}.
\begin{asss}\label{asss:1,2}
In the notation above, we impose the following:
\begin{enumerate}
\item\label{A.1} There exist constants $c_0,c_1>0$ such that for all $\lambda\in\a*$ with $|\lambda|\geq 1$,
	\begin{equation}\label{eq:A.1}
	c_0(1+\psi(\lambda))\leq q_1(\lambda) \leq c_1(1+\psi(\lambda)).
	\end{equation}
\item\label{A.2.M} Let $M\in\mathbb{N}$, $M>\dim(G/K)$, and suppose that $q_2(\cdot,\lambda)\in C^M_c(K|G|K)$ for all $\lambda\in\a*$. Suppose further that for $\beta=0,1,\ldots,M$, there exists $\Phi_\beta\in L^1(K|G|K)$ such that 
	\begin{equation}\label{eq:A.2.M}
	\left|(-\Delta)^{\beta/2}F_{\lambda,\eta}(\sigma)\right| \leq \Phi_\beta(\sigma)\langle\lambda\rangle^M\big(1+\psi(\eta)\big),
	\end{equation}
for all $\lambda,\eta\in\a*$, $\sigma\in G$, where $F_{\lambda,\eta}(\sigma)=\phi_{-\lambda}(\sigma)q_2(\sigma,\eta)$ (c.f.~(\ref{eq:F})).
\end{enumerate}
\end{asss}

\begin{rmks}
\begin{enumerate}
\item These assumptions are analogues to P.1, P.2.q of \citet{jacob94}, pp.~156, or (A.1), (A.2.M) of \citet{hoh}, pp.54.
\item As noted in Remark \ref{rmks:Φ_βlem} (\ref{calc}), the conditions in Assumption \ref{asss:1,2} (\ref{A.2.M}) imply that
	\begin{equation}\label{eq:q_2^wedge}
	(q_2(\sigma,D)u)^{\wedge}(\lambda) = \int_{\a*}\hat{F}_{\lambda,\eta}(-\eta)\hat{u}(\eta)\omega(d\eta),
	\end{equation}
for all $\lambda\in\a*$ and $u\in C_c^\infty(K|G|K)$, a fact that will be useful several times more.
\end{enumerate}
\end{rmks}
\begin{thm}\label{thm:H^psi,2}
Subject to Assumptions \ref{asss:1,2}, for all $s\in\mathbb{R}$, $q_1(D)$ extends to a continuous operator from $H^{\psi,s+2}$ to $H^{\psi,s}$, and $q(\sigma,D)$ extends to a continuous operator from $H^{\psi,2}$ to $L^2(K|G|K)$.
\end{thm}
\begin{proof}
The proof of the first part is omitted, since it is an easy adaptation of the proof of Theorem 4.8 on page 55 of \citet{hoh} --- first proved as Corollary 3.1 in \citet{jacob94}. 

The second part is also proved similarly to Theorem 4.8 of \citet{hoh}, the main difference being that $\hat{F}_{\lambda,\eta}(-\eta)$ takes the place of the transformed symbol, as discussed previously (see (\ref{eq:hat{q}})). By (\ref{eq:q_2^wedge}) and the Plancherel theorem,
	\begin{align*}
	|\langle q_2(\sigma,D)u,v\rangle| &= \left|\int_{\a*}(q_2(\sigma,D)u)^{\wedge}(\lambda)\overline{\hat{v}(\lambda)}\omega(d\lambda)\right| \\
	&=  \left|\int_{\a*}\int_{\a*}\hat{F}_{\lambda,\eta}(-\eta)\hat{u}(\eta)\overline{\hat{v}(\lambda)}\omega(d\eta)\omega(d\lambda)\right| \\
	&\leq \int_{\a*}\int_{\a*}\left|\hat{F}_{\lambda,\eta}(-\eta)\right||\hat{u}(\eta)||\hat{v}(\lambda)|\omega(d\eta)\omega(d\lambda). 
	\end{align*}
Then, using (\ref{eq:A.2.M}), Lemma \ref{lem:Φ_β} and Young's convolution inequality\footnote{See \citet{simon1} Theorem 6.6.3, page 550. Here, we are again identifying $\a*$ with a Euclidean space.},
	\begin{align*}
	|\langle q_2(\sigma,D)u,v\rangle| &\leq C_M\sum_{\beta=0}^M\Vert\Phi_\beta\Vert_1\int_{\a*}\int_{\a*}\langle\lambda-\eta\rangle^{-M}\Psi(\eta)^2|\hat{u}(\eta)||\hat{v}(\lambda)|\omega(d\eta)\omega(d\lambda) \\
	&= C_M\sum_{\beta=0}^M\Vert\Phi_\beta\Vert_1\int_{\a*}\left[\langle\cdot\rangle^{-M}\ast\big(\Psi^2|\hat{u}|\big)\right](\lambda)|\hat{v}(\lambda)|\omega(d\lambda) \\
	&\leq C_M\sum_{\beta=0}^M\Vert\Phi_\beta\Vert_1\left\Vert\langle\cdot\rangle^{-M}\ast\big(\Psi^2|\hat{u}|\big)\right\Vert_{L^2(\a*,\omega)}\Vert\hat{v}\Vert_{L^2(\a*,\omega)} \\
	&\leq C_M\sum_{\beta=0}^M\Vert\Phi_\beta\Vert_1\left\Vert\langle\cdot\rangle^{-M}\right\Vert_{L^1(\a*,\omega)}\Vert u\Vert_{\psi,2}\Vert v\Vert,
	\end{align*}
for all $u,v\in C_c^\infty(K|G|K)$. Hence, for all $u\in C_c^\infty(K|G|K)$,
	\begin{align*}
	\Vert q_2(\sigma,D)u\Vert &= \sup_{\substack{v\in C_c^\infty(K|G|K) \\ \Vert v\Vert=1}}|\langle q_2(\sigma,D)u,v\rangle| \leq C_M\sum_{\beta=0}^M\Vert\Phi_\beta\Vert_1\left\Vert\langle\cdot\rangle^{-M}\right\Vert_{L^1(\a*,\omega)}\Vert u\Vert_{\psi,2},
	\end{align*}
and $q_2(\sigma,D)$ extends to a bounded linear operator $H^{\psi,2}\to L^2(K|G|K)$.
\end{proof}

Under an additional assumption, we are able to obtain a more powerful result.
\begin{thm}\label{thm:H^psi,s}
Suppose Assumptions \ref{asss:1,2} hold, and suppose further that $s\in\mathbb{R}$ satisfies $|s-1|+1+\dim(G/K) < M$. Then $q(\sigma,D)$ extends to a continuous linear operator from $H^{\psi,s+2}\to H^{\psi,s}$.
\end{thm}

We first need a technical lemma.
\begin{lem}\label{lem:H^psi,s}
Let $s\in\mathbb{R}$ and $M\in\mathbb{N}$ be such that $|s-1|+1+\dim(G/K)<M$. Then for all $\lambda,\eta\in\a*$,
	\begin{equation}\label{eq:claim}
	\left|\Psi(\lambda)^s-\Psi(\eta)^s\right| \leq C_{s,\psi}\langle\lambda-\eta\rangle^{|s-1|+1}\Psi(\eta)^{s-1},
	\end{equation}
where
	\begin{equation}\label{eq:C_s,phi}
	C_{s,\psi} = 2^{(|s-1|+2)/2}(1+c_\psi)^{(|s-1|+1)/2}|s|,
	\end{equation}
and  $c_\psi$ is the constant from Proposition \ref{prop:negdef} (\ref{c_phi}).
\end{lem}
\begin{proof}
This is a special case of a bound obtained in \citet{hoh} --- see page 50, lines 5--11.
\end{proof}

\begin{proof}[Proof of Theorem \ref{thm:H^psi,s}]
By Theorem \ref{thm:H^psi,2}, it suffices to prove that $q_2(\sigma,D)$ extends to a continuous operator from $H^{\psi,s+2}\to H^{\psi,s}$. Given $u\in C^\infty_c(K|G|K)$,
	\begin{equation}\label{eq:calc1}
	\begin{aligned}
	\Vert q_2(\sigma,D)u\Vert_{\psi,s} &= \Vert\Psi(D)^sq_2(\sigma,D)u\Vert \\
	&\leq \Vert q_2(\sigma,D)\Psi(D)^su\Vert + \Vert[\Psi(D)^s,q_2(\sigma,D)]u\Vert.
	\end{aligned}
	\end{equation}
Also, by Theorem \ref{thm:H^psi,2} and Theorem \ref{thm:aniso} (\ref{ctsembed}),
	\begin{equation}\label{eq:calc2}
	\Vert q_2(\sigma,D)\Psi(D)^su\Vert \leq C\Vert\Psi(D)^su\Vert_{\psi,2} = C\Vert u \Vert_{\psi,s+2},
	\end{equation}
where $C = C_M\sum_{\beta=0}^M\Vert\Phi_\beta\Vert_1\left\Vert\langle\cdot\rangle^{-M}\right\Vert_{L^1(\a*,\omega)}$. We will estimate
	\[\left\Vert[\Psi(D)^s,q_2(\sigma,D)]u\right\Vert,\]
Our method is similar to that in Theorem 4.3 of \citet{hoh}, and so some details are omitted. The map $F_{\lambda,\eta}$ replaces the transformed symbol $\hat{q}$ once again. 

One can check using (\ref{eq:q_2^wedge}) that for all $\lambda\in\a*$,
	\[([\Psi(D)^s,q_2(\sigma,D)]u)^{\wedge}(\lambda) = \int_{\a*}\hat{F}_{\lambda,\eta}(-\eta)\big\{\Psi(\lambda)^s-\Psi(\eta)^s\big\}\hat{u}(\eta)\omega(d\eta),\]
and hence for all $u,v\in C_c^\infty(K|G|K)$,
	\[\left|\big\langle[\Psi(D)^s,q_2(\sigma,D)]u,v\big\rangle\right| \leq \int_{\a*}\int_{\a*}\left|\hat{F}_{\lambda,\eta}(-\eta)\right|\left|\Psi(\lambda)^s-\Psi(\eta)^s\right||\hat{u}(\eta)||\hat{v}(\lambda)|\omega(d\eta)\omega(d\lambda).\]
By Lemmas \ref{lem:Φ_β} and \ref{lem:H^psi,s}, 
	\begin{align*}	
	&\left|\big\langle[\Psi(D)^s,q_2(\sigma,D)]u,v\big\rangle\right| \\
	&\hspace{50pt}\leq C_{s,\psi,M}\int_{\a*}\int_{\a*}\langle\lambda-\eta\rangle^{-M+|s-1|+1}\Psi(\eta)^{s+1}|\hat{u}(\eta)||\hat{v}(\lambda)|\omega(d\eta)\omega(d\lambda) \\
	&\hspace{50pt}= C_{s,\psi,M}\int_{\a*}\left(\langle\cdot\rangle^{-M+|s-1|+1}\ast\left[\Psi^{s+1}|\hat{u}|\right]\right)(\lambda)|\hat{v}(\lambda)|\omega(d\lambda),
	\end{align*}
where $C_{s,\psi,M} = C_{s,\psi}C_M\sum_{\beta=0}^M\Vert\Phi_\beta\Vert_1$. By Lemma \ref{lem:<.>}, $\langle\cdot\rangle^{-(M-|s-1|-2)}\in L^1(\a*,\omega)$, and one can check using the Cauchy--Schwarz and Young inequalities that
	\[\left|\big\langle[\Psi(D)^s,q_2(\sigma,D)]u,v\big\rangle\right| \leq C_{s,\psi,M}\left\Vert\langle\cdot\rangle^{-(M-|s-1|-1)}\right\Vert_{L^1(\a*,\omega)}\Vert u\Vert_{\psi,s+1}\Vert v\Vert. \]
Taking the supremum over $v\in C_c^\infty(K|G|K)$, with $\Vert v\Vert=1$,
	\[\Vert[\Psi(D)^s,q_2(\sigma,D)]u\Vert \leq C_{s,\psi,M}\left\Vert\langle\cdot\rangle^{-(M-|s-1|-1)}\right\Vert_{L^1(\a*,\omega)}\Vert u\Vert_{\psi,s+1}.\]
Combining with (\ref{eq:calc1}) and (\ref{eq:calc2}),
	\begin{equation}\label{eq:q2est}
	\Vert q_2(\sigma,D)u\Vert_{\psi,s} \leq C_M\sum_{\beta=0}^M\Vert\Phi_\beta\Vert_{L^1(\a*,\omega)}\Big(\left\Vert\langle\cdot\rangle^{-M}\right\Vert_{L^1(\a*,\omega)}\Vert u\Vert_{\psi,s+2} + C_{s,\psi}\Vert u\Vert_{\psi,s+1}\Big).
	\end{equation}
Theorem \ref{thm:aniso} (\ref{ctsembed}) may now be used to obtain the desired bound.
\end{proof}

To prove (\ref{eq:hyr3}), we seek solutions $u$ to the equation
	\begin{equation}\label{eq:dense}
	(q(\sigma,D)+\alpha)u=f,
	\end{equation}
for a given function $f$ and $\alpha>0$. Consider the bilinear form $B_\alpha$ defined by 
	\[B_\alpha(u,v) = \langle(q(\sigma,D)+\alpha)u,v\rangle, \hspace{20pt} \forall u,v\in C_c^\infty(K|G|K).\]

\begin{thm}
Suppose Assumptions \ref{asss:1,2} hold with $M>\dim(G/K)+1$. Then $B_\alpha$ extends continuously to $H^{\psi,1}\times H^{\psi,1}$.
\end{thm}
\begin{proof}
This proof is very similar to those of \citet{jacob94} Lemma 3.2, pp.~160, and \citet{hoh} Theorem 4.9, pp.~56, and so we give only a sketch. 

Let $u,v\in H^{\psi,1}$. Using Assumption \ref{asss:1,2} (\ref{A.1}) and the fact that $q_1$ is continuous, there is $\kappa_1>0$ such that $|q_1|\leq\kappa_1\Psi^2$. Plancherel's identity may then be used to show that
	\[\left|\langle q_1(D)u,v\rangle\right| \leq \int_{\a*}|q_1(\lambda)||\hat{u}(\lambda)||\hat{v}(\lambda)|\omega(d\lambda) \leq \kappa_1\Vert u\Vert_{\psi,1}\Vert v\Vert_{\psi,1}.\]
Furthermore, methods similar to the proof of Theorem \ref{thm:H^psi,s} are used to show that
	\begin{equation}\label{eq:q_2}
	\left|\langle q_2(\sigma,D)u,v\rangle\right| \leq \kappa_2\left\Vert\langle\cdot\rangle^{-M+1}\right\Vert_{L^1(\a*,\omega)}\Vert u\Vert_{\psi,1}\Vert v\Vert_{\psi,1},
	\end{equation}
where 
	\begin{equation}\label{eq:kappa2}
	\kappa_2=C_M\sqrt{2(1+c_\psi)}\sum_{\beta=0}^M\Vert\Phi_\beta\Vert_{L^1(\a*,\omega)}.
	\end{equation}
By Theorem \ref{thm:aniso} (\ref{ctsembed}), there is $\kappa_3>0$ such that $\Vert u\Vert\leq \kappa_3\Vert u\Vert_{\psi,1}$, and thus
	\[\left|B_\alpha(u,v)\right| \leq \left|\langle q_1(D)u,v\rangle\right| + \left|\langle q_2(\sigma, D)u,v\rangle\right| + \alpha\left|\langle u,v\rangle\right| \leq \left(\kappa_1 + \kappa_2 + \alpha\kappa_3^2\right)\Vert u\Vert_{\psi,1}\Vert v\Vert_{\psi,2},\]
for all $u,v\in H^{\psi,1}$, which proves the theorem.
\end{proof}

The following assumption will ensure that for $\alpha$ sufficiently large, $B_\alpha$ is coercive on $H^{\psi,1}$. We will then use the Lax--Milgram theorem to obtain a weak solution to (\ref{eq:dense}).
\begin{ass}\label{ass:3}
Let $M\in\mathbb{N}$, $M>\dim(G/K)+1$, and write
	\[\gamma_M = \left(8C_M(2(1+c_\psi))^{1/2}\Vert\langle\cdot\rangle^{-M+1}\Vert_{L^1(\a*,\omega)}\right)^{-1},\]
where $c_{\psi}$ and $C_M$ are constants given by (\ref{eq:c_phi}) and  (\ref{eq:C_M}), respectively.

For $c_0$ is as in Assumption \ref{asss:1,2} (\ref{A.1}), assume that 
	\[\sum_{\beta=0}^M\Vert\Phi_\beta\Vert_1 \leq \gamma_Mc_0.\]
\end{ass}
\begin{rmk}
See \citet{jacob94} P.3 and P.4, pp.~161, or \citet{hoh} (A.3.M), pp.~54, for comparison. Examples where Assumption \ref{ass:3} is satisfied are considered in Section \ref{sec:EGs}.
\end{rmk}

The next theorem is an analogue of Theorem 3.1 of \citet{jacob94}.
\begin{thm}\label{thm:coercive}
Suppose Assumptions \ref{asss:1,2} and \ref{ass:3} hold, with $M>\dim(G/K)+1$. Then there is $\alpha_0>0$ such that
	\[B_\alpha(u,u)\geq\frac{c_0}{2}\Vert u\Vert_{1,\lambda}^2,\]
for all $\alpha\geq\alpha_0$ and $u\in H^{\psi,1}$. In particular, $B_\alpha$ is coercive for all $\alpha\geq\alpha_0$.
\end{thm}
\begin{proof}
Proceed exactly as in \citet{hoh} page 57, lines 8--17. By Assumption \ref{asss:1,2} (\ref{A.1}), there is $\alpha_0>0$ such that
	\begin{equation}\label{eq:alpha0}
	q_1(\lambda) \geq  c_0\Psi(\lambda)^2-\alpha_0 \hspace{20pt} \forall\lambda\in\a*.
	\end{equation}
This may be used to prove that for all $u\in H^{\psi,1}$,
	\[\langle q_1(D)u,u\rangle \geq c_0\Vert u\Vert_{\psi,1}^2 - \alpha_0\Vert u\Vert^2,\]
at which point we can apply (\ref{eq:kappa2}) and (\ref{eq:q_2}), as well as Assumption \ref{ass:3}, to conclude
	\begin{align*}
	\left|\langle q_2(\sigma,D)u,u\rangle\right| &\leq C_M\sqrt{2(1+c_\psi)}\sum_{\beta=0}^M\Vert\Phi_\beta\Vert_{L^1(\a*,\omega)}\left\Vert\langle\cdot\rangle^{-M+1}\right\Vert_{L^1(\a*,\omega)}\Vert u\Vert_{\psi,1}^2 \\
	&= \frac{1}{8\gamma_M}\sum_{\beta=0}^M\Vert\Phi_\beta\Vert_{L^1(\a*,\omega)}\Vert u\Vert_{\psi,1}^2 \leq \frac{c_0}{8}\Vert u\Vert_{\psi,1},
	\end{align*}
for all $u\in H^{\psi,1}$. Thus, for all $u\in H^{\psi,1}$
	\begin{align*}
	\langle q(\sigma,D)u,u\rangle &\geq \langle q_1(D)u,u\rangle - \left|\langle q_2(\sigma,D)u,u\rangle\right| \\
	&\geq (c_0-\frac{c_0}{8})\Vert u\Vert_{\psi,1}^2 - \alpha_0\Vert u\Vert_{\psi,1}^2 \geq \frac{c_0}{2}\Vert u\Vert_{\psi,1}^2 - \alpha_0\Vert u\Vert_{\psi,1}^2.
	\end{align*}
Therefore, for all $\alpha\geq\alpha_0$ and $u\in H^{\psi,1}$
	\begin{align*}
	B_\alpha(u,u) &= \langle q(\sigma,D)u,u\rangle + \alpha\Vert u\Vert \geq \langle q(\sigma,D)u,u\rangle + \alpha_0\Vert u\Vert \geq \frac{c_0}{2}\Vert u\Vert_{\psi,1}^2,
	\end{align*}
\end{proof}

\begin{thm}\label{thm:weaksoln}
Let $\alpha\geq\alpha_0$. Then (\ref{eq:dense}) has a weak solution in the following sense: for all $f\in L^2(K|G|K)$ there is a unique $u\in H^{\psi,1}$ such that for all $v\in H^{\psi,1}$,
	\[B_\alpha(u,v) = \langle f,v\rangle.\]
\end{thm}
\begin{proof}
Apply the Lax--Milgram theorem (Theorem 1 of \citet{evans}, pp.~297) to $B_\alpha$, using the linear functional $v\mapsto\langle f,v\rangle$.
\end{proof}

Having found a weak solution to (\ref{eq:dense}), the next task is to prove that this solution is in fact a strong solution that belongs to $C_0(K|G|K)$. This will be achieved using the Sobolev embedding of Theorem \ref{thm:aniso} (\ref{SobEmbed}). 

Just as in \citet{jacob94} Theorem 3.1 and \citet{hoh} Theorem 4.11, we have a useful lower bound for the pseudodifferential operator $q(\sigma,D)$ acting on $H^{\psi,s}$, when $s\geq 0$.
\begin{thm}\label{thm:l.b.}
Let $s\geq 0$, and suppose the symbol $q$ satisfies Assumptions \ref{asss:1,2} and \ref{ass:3}, for some $M>|s-1|+1+\dim(G/K)$. Then there is $\kappa>0$ such that for all $u\in H^{\psi,s+2}$,
	\[\Vert q(\sigma,D)u\Vert_{\psi,s} \geq \frac{c_0}{4}\Vert u\Vert_{\psi,s+2}-\kappa\Vert u \Vert.\]
\end{thm}
\begin{proof}
The proof is formally no different to the sources mentioned: let $u\in H^{\psi,s+2}$, and use  (\ref{eq:alpha0}) and Theorem \ref{thm:aniso} (\ref{s3s2s1}) to prove that 
	\begin{equation}\label{eq:q1}
	\Vert q_1(D)u\Vert_{\psi,s} \geq \frac{c_0}{2}\Vert u\Vert_{\psi,s+2} - \kappa_1\Vert u\Vert,
	\end{equation}
for some $\kappa_1>0$. Recall the estimate (\ref{eq:q2est}) of $\Vert q_2(\sigma,D)u\Vert_{\psi,s}$ from the proof of Theorem \ref{thm:H^psi,s}. In light of Assumption \ref{ass:3} and the particular form chosen for $\gamma_M$, one can use (\ref{eq:q2est}) to show that
	\begin{align*}
	\Vert q_2(\sigma,D)u\Vert_{\psi,s} &\leq C_M\sum_{\beta=0}^M\Vert\Phi_\beta\Vert_{L^1(\a*,\omega)}\Big(\left\Vert\langle\cdot\rangle^{-M}\right\Vert_{L^1(\a*,\omega)}\Vert u\Vert_{\psi,s+2} + C_{s,\psi}\Vert u\Vert_{\psi,s+1}\Big) \\
	&\leq C_Mc_0\gamma_M\Big(\left\Vert\langle\cdot\rangle^{-M}\right\Vert_{L^1(\a*,\omega)}\Vert u\Vert_{\psi,s+2} + C_{s,\psi}\Vert u\Vert_{\psi,s+1}\Big) \\
	&\leq \frac{c_0}{8}\Vert u\Vert_{\psi,s+2} + c\Vert u\Vert_{\psi,s+1},
	\end{align*}
where $c>0$ is a constant. Using Theorem \ref{thm:aniso} (\ref{s3s2s1}) once again, let $\kappa_2>0$ such that
	\[c\Vert u\Vert_{\psi,s+1} \leq \frac{c_0}{8}\Vert u\Vert_{\psi,s+2} + \kappa_2\Vert u\Vert.\]
	Then, by the above,
	\begin{equation}\label{eq:q2}
	\Vert q_2(\sigma,D)u\Vert_{\psi,s} \leq \frac{c_0}{4}\Vert u\Vert_{\psi,s+2} + \kappa_2\Vert u\Vert.
	\end{equation}
Combining (\ref{eq:q1}) and (\ref{eq:q2}), we get 
	\[\Vert q(\sigma,D)u\Vert_{\psi,s} \geq \Vert q_1(D)u\Vert_{\psi,s} - \Vert q_2(\sigma,D)u\Vert_{\psi,s} \geq \frac{c_0}{4}\Vert u\Vert_{\psi,s+2} - (\kappa_1+\kappa_2)\Vert u\Vert.\]
\end{proof}

The proof of the next theorem makes use of a particular family $(J_\epsilon,0<\epsilon\leq 1)$ of bounded linear operators on $L^2(K|G|K)$, which will play the role of a Friedrich mollifier, but in the noncompact symmetric space setting. 

First note that by identifying $\al$ with $\mathbb{R}^m$ via our chosen basis, it makes sense to consider Friedrich mollifiers on $\al$. For $0<\epsilon\leq 1$ and $H\in\al$, let
	\[l(H) := C_0e^{\frac{1}{|H|^2-1}}\ind_{B_1(0)}(H), ~~\text{ and }~~ l_\epsilon(H) := \epsilon^{-m}l(H/\epsilon),\]
where $C_0>0$ is a constant chosen so that $\int_{\al}l(H)dH=1$. This mollifier is used frequently in \citet{evans} (see Appendix C.4, pp.~629), and \citet{jacob94} and \citet{hoh} use it to pass from a weak solution result to a strong solution result.

Observe that $l,l_\epsilon\in\mathcal{S}(\al)^W$ for all $0<\epsilon\leq 1$. Using Theorem \ref{thm:ComDiag}, let $j,j_\epsilon\in\mathcal{S}(K|G|K)$ be such that 
	\[\hat{j} = \mathscr{F}(l), \hspace{5pt} \text{ and } \hspace{5pt} \hat{j_\epsilon} = \mathscr{F}(l_\epsilon), \hspace{20pt} \forall 0<\epsilon\leq 1,\]
where $\mathscr{F}$ denotes the Euclidean Fourier transform (see equation (\ref{eq:EucF})). For $0<\epsilon\leq 1$, let $J_\epsilon$ be the convolution operator defined on $L^2(K|G|K)$ by
	\[J_\epsilon u = j_\epsilon\ast u \hspace{20pt} \forall f\in L^2(K|G|K).\]

The most important properties of $(J_\epsilon,0<\epsilon\leq 1)$ needed for the	proof of Theorem \ref{thm:strong} are stated below, and proven in the Section \ref{sec:PfofProp}.
\begin{prop}\label{prop:Fmoll}
\begin{enumerate}
\item\label{epslam} $\hat{j_\epsilon}(\lambda)=\hat{j}(\epsilon\lambda)$ for all $0<\epsilon\leq 1$ and $\lambda\in\a*$.
\item\label{s-a} For all $0<\epsilon\leq 1$, $J_\epsilon$ is a self-adjoint contraction of $L^2(K|G|K)$.
\item $J_\epsilon u\in H^{\psi,s}$ for all $s\geq 0$, $u\in L^2(K|G|K)$ and $0<\epsilon\leq 1$, and if $u\in H^{\psi,s}$, then
	\[\Vert J_\epsilon u\Vert_{\psi,s} \leq \Vert u\Vert_{\psi,s}.\]
\item\label{J_epsu-u} For all $s\geq 0$ and $u\in H^{\psi,s}$, $\Vert J_\epsilon u - u\Vert_{\psi,s}\rightarrow 0$ as $\epsilon\rightarrow 0$.
\end{enumerate}
\end{prop}

The following commutator estimate will also be useful in the proof of Theorem \ref{thm:strong}.
\begin{lem}\label{lem:commutator}
Let $s\geq 0$, and suppose $q$ is a continuous negative definite symbol satisfying Assumption \ref{asss:1,2} (\ref{A.2.M}) for $M>|s-1|+1+\dim(G/K)$. Then there is $c>0$ such that for all $ 0<\epsilon\leq 1$ and all $u\in C_c^\infty(K|G|K)$,
	\[\Vert[J_\epsilon,q(\sigma,D)]u\Vert_{\psi,s} \leq c\Vert u\Vert_{\psi,s+1}.\]
\end{lem}
\begin{proof} 
Let $0<\epsilon\leq 1$ and $u\in C_c^\infty(K|G|K)$, and observe that by Proposition \ref{prop:Fmoll} (\ref{s-a}),
	\[[J_\epsilon,q_1(D)]u)^\wedge(\lambda) = \hat{j}(\epsilon\lambda)q_1(\lambda)\hat{u}(\lambda) - q_1(\lambda)\hat{j}(\epsilon\lambda)\hat{u}(\lambda) = 0,\]
for all $\lambda\in\a*$, so $[J_\epsilon,q_1(D)]u = 0$. For $\lambda,\eta\in\a*$, let $F_{\lambda,\eta}=\phi_{-\lambda}q_2(\cdot,\eta)$, as previously (c.f.~(\ref{eq:F})). Then by (\ref{eq:calc}) and Proposition \ref{prop:Fmoll} (\ref{epslam}), for all $\lambda\in\a*$,
	\begin{align*}
	\left([J_\epsilon,q(\sigma,D)]u\right)^\wedge(\lambda) &= (J_\epsilon q_2(\sigma,D)u)^\wedge(\lambda) - (q_2(\sigma,D)J_\epsilon u)^\wedge(\lambda) \\
	&=\hat{j}(\epsilon\lambda)(q_2(\sigma,D)u)^\wedge(\lambda) - \int_{\a*}\hat{F}_{\lambda,\eta}(-\eta)(J_\epsilon u)^\wedge(\eta)\omega(d\eta) \\
	&=\hat{j}(\epsilon\lambda)(q_2(\sigma,D)u)^\wedge(\lambda) - \int_{\a*}\hat{F}_{\lambda,\eta}(-\eta)\hat{j}(\epsilon\eta)\hat{u}(\eta)\omega(d\eta).
	\end{align*}
Applying (\ref{eq:calc}) once more,
	\begin{equation}\label{eq:calc'}
	\left([J_\epsilon,q(\sigma,D)]u\right)^\wedge(\lambda) = \int_{\a*}\hat{F}_{\lambda,\eta}(-\eta)\left(\hat{j}(\epsilon\lambda)-\hat{j}(\epsilon\eta)\right)\hat{u}(\eta)\omega(d\eta),
	\end{equation}
for all $\lambda\in\a*$. From here, a straightforward adaptation to the proof of \citet{hoh} Theorem 4.4, pp.~51--52, with (\ref{eq:calc'}) replacing \citet{hoh} (4.23), completes the proof of the lemma.
\end{proof}

We are now ready to state and prove that, subject to our conditions, a strong solution to (\ref{eq:dense}) exists, and belongs to an anisotropic Sobolev space of suitably high order.
\begin{thm}\label{thm:strong}
Let $\alpha_0$ be as in Theorem \ref{thm:coercive}, let $\alpha\geq\alpha_0$, and let $s\geq 0$. Suppose that the continuous negative definite symbol $q$ satisfies Assumptions \ref{asss:1,2} and \ref{ass:3}, where $M>|s-1|+1+\dim(G/K)$. Then for all $f\in H^{\psi,s}$, there is a unique $u\in H^{\psi,s+2}$ such that
	\begin{equation}\label{eq:strong}
	(q(\sigma,D)+\alpha)u = f.
	\end{equation}
\end{thm}
\begin{proof}
Let $f\in H^{\psi,s}$. By Theorem \ref{thm:aniso} we also have $f\in L^2(K|G|K)$, and so by Theorem \ref{thm:weaksoln} there is a unique $u\in H^{\psi,1}$ such that 
	\begin{equation}\label{eq:weak}
	B_\alpha(u,v) =\langle f,v\rangle \hspace{20pt} \forall v\in C_c^\infty(K|G|K).
	\end{equation}
The proof follows that of \citet{jacob94} Theorem 4.3, pp.~163 and \citet{hoh} Theorem 4.12, pp.~59, using induction to show that that $u\in H^{\psi,t}$ for $1\leq t\leq s+2$, and in particular, that $u\in H^{\psi,s+2}$. The family of operators  $(J_\epsilon,0<\epsilon\leq 1)$ take over role of the Friedrich mollifiers of \citet{jacob94} and \citet{hoh}. By Proposition \ref{prop:Fmoll} these operators satisfy the properties needed for the proof to carry over with little alteration. Lemma \ref{lem:commutator} and Theorem \ref{thm:l.b.} replace \citet{hoh} Theorem 4.4 and 4.11, respectively.
\end{proof}

\begin{thm}\label{thm:hyrpt3}
Let $q$ be a continuous negative definite symbol, satisfying Assumptions \ref{asss:1,2} and \ref{ass:3} with $M>\max\left\{1,\frac{d}{r}\right\}+d$, where $d=\dim(G/K)$. Then for all $\alpha\geq \alpha_0$, 
	\[\overline{\Ran(\alpha+q(\sigma,D))}=C_0(K|G|K).\]
\end{thm}
\begin{proof}
Fix $s\in\mathbb{R}$ with $\max\left\{\frac{d}{r},1\right\}< s<M-d$. Let $\mathcal{A}$ denote the linear operator on $C_0(K|G|K)$ with domain $H^{\psi,s+2}$, defined by $\mathcal{A}u = -q(\sigma,D)u$ for all $u\in\Dom(\mathcal{A})$. By a similar argument to that on  page 60 of \citet{hoh}, one can show using that $C_c^\infty(K|G|K)$ is a operator core for $\mathcal{A}$, with
	\[\overline{\Ran(\alpha+q(\sigma,D))} = \overline{\Ran(\alpha-\mathcal{A})}\]
for all $\alpha\in\mathbb{R}$. Here, Theorem \ref{thm:aniso} (\ref{SobEmbed}) replaces \citet{hoh} Proposition 4.1, and Theorem \ref{thm:H^psi,s} replaces \citet{hoh} Theorems 4.8 and 4.11.

Let $\alpha_0$ be as in Theorem \ref{thm:strong}. We show that $\overline{\Ran(\alpha-\mathcal{A})}=C_0(K|G|K)$ for all $\alpha\geq \alpha_0$. Given $f\in C_0(K|G|K)$, choose a sequence $(f_n)$ in $H^{\psi,s}$ such that $\Vert f_n-f\Vert_\infty\rightarrow 0$ as $n\rightarrow\infty$. Then $f_n\in\Ran(\alpha-\mathcal{A})$ for all $\alpha\geq\alpha_0$, and thus $f\in\overline{\Ran(\alpha-\mathcal{A})}$ for all $\alpha\geq\alpha_0$.
\end{proof}

Combining Theorem \ref{thm:hyrpt3} with the work of Section \ref{sec:GangOps&HYR} yields the following.
\begin{cor}\label{cor:sub-Mark}
Let $q$ be a Gangolli symbol that satisfies Assumptions \ref{asss:1,2} and \ref{ass:3} for some $M>\min\{1,d/r\}+d$. Then $-q(\sigma,D)$ extends to the infinitesimal generator of a strongly continuous sub-Feller semigroup on $C_0(K|G|K)$.
\end{cor}
\begin{proof}
By construction, $-q(\sigma,D)$ is a densely defined linear operator on $C_0(K|G|K)$. It is a Gangolli operator, and hence satisfies the positive maximum principle. By Theorems \ref{thm:hyr} and \ref{thm:hyrpt3}, $-q(\sigma,D)$ is closable, and its closure generates a strongly continuous sub-Feller semigroup.
\end{proof}
%
\section{A Class of Examples}\label{sec:EGs}
We now present a class of Gangolli symbols that satisfy the conditions of Corollary \ref{cor:sub-Mark}. Let $M\in\mathbb{N}$ such that $M>\min\{1,d/r\}+d+1$. We consider symbols $q:G\times\a*\to\mathbb{R}$ of the form
	\begin{equation}\label{eq:qEG}
	q(\sigma,\lambda) = \kappa\psi(\lambda) + u(\sigma)v(\lambda), \hspace{20pt} \forall \sigma\in G,\lambda\in\a*,
	\end{equation}
where $\kappa$ is a positive constant, $\psi:\a*\to\mathbb{R}$ is a Gangolli exponent satisfying (\ref{eq:psi_est}), $u\in C^M_c(K|G|K)$ is non-negative, and $v:\a*\to\mathbb{R}$ is a Gangolli exponent satisfying, for some $c_v>0$,
	\begin{equation}\label{eq:vbnd}
	|v(\lambda)|\leq c_v(1+\psi(\lambda)) \hspace{20pt} \forall\lambda\in\a*.
	\end{equation} 

By Example \ref{eg:GangSymbs}, the mappings $(\sigma,\lambda)\mapsto c_0\psi(\lambda)$ and $(\sigma,\lambda)\mapsto u(\sigma)v(\lambda)$ are both Gangolli symbols, and hence so is $q$.

For each $\lambda\in \a*$ and $\sigma\in G$, let
	\begin{equation}\label{eq:q_1q_2EG}
	q_1(\lambda) = \kappa\psi(\lambda), \hspace{10pt} \text{ and } \hspace{10pt} q_2(\sigma,\lambda) = u(\sigma)v(\lambda).
	\end{equation}
Observe that $q$ is of the form (\ref{eq:q}): since $v$ has compact support, $\Supp(v)\neq G$, and if $\sigma_0\in G\setminus\Supp(v)$, then $q_1=q(\sigma_0,\cdot)$. 

\begin{prop}
$q_1$ satisfies Assumption \ref{asss:1,2} (\ref{A.1}).
\end{prop}
\begin{proof}
The upper bound of (\ref{eq:A.1}) may be easily verified by taking $c_1=\kappa$. For the lower bound, suppose $|\lambda|\geq 1$. Then by (\ref{eq:psi_est}), 
	\[q_1(\lambda) = \frac{\kappa}{2}(\psi(\lambda)+\psi(\lambda)) \geq \frac{\kappa}{2}(c|\lambda|^r+\psi(\lambda)) \geq \frac{\kappa}{2}\min\{1,c\}(1+\psi(\lambda)),\]
and so taking $c_0=\frac{\kappa}{2}\min\{1,c\}$, the result follows.
\end{proof}

For Assumption \ref{asss:1,2} (\ref{A.2.M}), note that in the case we are considering,
	\[F_{\lambda,\eta}(\sigma) = \phi_{-\lambda}(\sigma)u(\sigma)v(\eta), \hspace{20pt} \forall\sigma\in G,\;\lambda,\eta\in\a*,\] 
and so, for $\beta=0,1,\ldots,M$,
	\[(-\Delta)^{\beta/2}F_{\lambda,\eta}(\sigma) = v(\eta)(-\Delta)^{\beta/2}(\phi_{-\lambda}u)(\sigma),\]
for all $\lambda,\eta\in\a*$ and $\sigma\in G$. By (\ref{eq:vbnd}),
	\[\left|(-\Delta)^{\beta/2}F_{\lambda,\eta}(\sigma)\right| = |v(\eta)|\left|(-\Delta)^{\beta/2}(\phi_{-\lambda}u)(\sigma)\right| \leq c_v\left|(-\Delta)^{\beta/2}(\phi_{-\lambda}u)(\sigma)\right|\big(1+\psi(\eta)\big).\]
For each $n\in\mathbb{N}$, a noncommutative version of the multinomial theorem tells us that
	\begin{equation}\label{eq:(-Delta)^n}
	(-\Delta)^n(\phi_{-\lambda}u) = (-1)^n\left(\sum_{j=1}^dX_j^2\right)^n(\phi_{-\lambda}u) = \sum_{\substack{\alpha\in\mathbb{N}_0^d,\\|\alpha|\leq r}}c_\alpha X^{\alpha}(\phi_{-\lambda}u)
	\end{equation}
for some coefficients $c_\alpha$, where $|\alpha|=\alpha_1+\ldots+\alpha_d$ and $X^{\alpha}:=X_1^{\alpha_1}\ldots X_d^{\alpha_d}$. 
Expanding the right-hand side of (\ref{eq:(-Delta)^n}) using the fact that each $X_j$ is a derivation will give a large sum of terms of the form
	\[\kappa_{X,Y}X\phi_{-\lambda}Yu,\]
where the $\kappa_{X,Y}$ are constants, and $X,Y\in{\bf D}(G)$ are products of powers of $X_1,\ldots, X_d$, each with degree at most $2n$. Let $\mathscr{U}_n$ be the set of all the $X$'s and $\mathscr{V}_n$ the set of all the $Y$'s, so that
	\begin{equation}\label{eq:X,Y}
	(-\Delta)^n(\phi_{-\lambda}u) = \sum_{\substack{X\in \mathscr{U}_n,\\Y\in \mathscr{V}_n}}\kappa_{X,Y}X\phi_{-\lambda}Yu.
	\end{equation}

The following bound will be useful.
\begin{lem}\label{lem:Xϕ}
For all $X\in{\bf D}(G)$, there is a constant $C_X>0$ such that 
	\begin{equation}\label{eq:C_X}
	|X\phi_\lambda(\sigma)| \leq C_X\langle\lambda\rangle^{\deg X}\phi_0(\sigma),
	\end{equation}
for all $\lambda\in\a*$ and $\sigma\in G$.
\end{lem}
\begin{proof}
This is a straightforward corollary of Theorem 1.1 (iii) of \citep[supplementary notes]{helg2SuppNotes} --- see also  \citet{HCI}, Lemma 46, pp.~294.
\end{proof}
 
\begin{prop}\label{prop:A.2.MEG}
The mapping $q_2$ in (\ref{eq:q_1q_2EG}) satisfies Assumption \ref{asss:1,2} (\ref{A.2.M}).
\end{prop}
\begin{proof}
It is clear by construction that $q_2(\cdot,\lambda)\in C_c^M(K|G|K)$ for all $\lambda\in\a*$.

To verify the rest of Assumption \ref{asss:1,2} (\ref{A.2.M}), it will be useful to assume that $M$ is even. Note that this is an acceptable assumption, since if $M$ is odd, we may replace it with $M-1$ --- the conditions of Corollary \ref{cor:sub-Mark} will still be satisfied.
Let $\beta\in\{0,1,\ldots,M\}$. We seek $\Phi_\beta\in L^1(K|G|K)$ for which 
	\begin{equation}\label{eq:A.2.M'}
	\left|(-\Delta)^{\beta/2}(\phi_{-\lambda}u)(\sigma)\right| \leq \Phi_\beta\langle\lambda\rangle^M, \hspace{20pt} \forall\sigma\in G,\;\lambda\in\a*.
	\end{equation}
Let $n=\lfloor\beta\rfloor$. Assume first that $\beta$ is even, so that $n=\beta/2$. By (\ref{eq:X,Y}) and Lemma \ref{lem:Xϕ},
	\begin{align*}
	\left|(-\Delta)^{\beta/2}(\phi_{-\lambda}u)\right| \leq  \sum_{\substack{X\in \mathscr{U}_n,\\Y\in \mathscr{V}_n}}|\kappa_{X,Y}||X\phi_{-\lambda}||Yu| 	&\leq \sum_{\substack{X\in \mathscr{U}_n,\\Y\in \mathscr{V}_n}}C_X|\kappa_{X,Y}||Yu|\langle \lambda\rangle^{\deg X}|\phi_0| \\
	&\leq \sum_{\substack{X\in \mathscr{U}_n,\\Y\in \mathscr{V}_n}}C_X|\kappa_{X,Y}||Yu|\langle \lambda\rangle^{\deg X},
	\end{align*}
since $|\phi_0|\leq 1$. Now, $\deg X\leq 2n=\beta\leq M$ for all $X\in\mathscr{U}_n$, and therefore,
	\[\left|(-\Delta)^{\beta/2}(\phi_{-\lambda}u)\right| \leq \kappa_\beta\sum_{Y\in \mathscr{V}_{\beta/2}}|Yu|\langle \lambda\rangle^M\]
where 
	\[\kappa_\beta = \sup\left\{C_X|\kappa_{X,Y}|:X\in \mathscr{U}_{\beta/2},Y\in \mathscr{V}_{\beta/2}\right\}.\]
Let
	\begin{equation}\label{eq:Phi_even}
	\Phi_\beta:=\kappa_\beta\sum_{Y\in \mathscr{V}_{\beta/2}}|Yu|.
	\end{equation}
Then $\Phi_\beta\in L^1(K|G|K)$, since each $Yu$ is a continuous function of compact support. Moreover,
	\begin{equation}\label{eq:L1even}
	\Vert\Phi_\beta\Vert_1 \leq \kappa_\beta\sum_{Y\in\mathscr{V}_{\beta/2}}\Vert Yu\Vert_1
	\end{equation}
In particular, we have verified (\ref{eq:A.2.M'}) when $\beta$ is even.    

Assume now that $\beta$ is odd, so that $(-\Delta)^{\beta/2}=\sqrt{-\Delta}(-\Delta)^n$. Since $M$ is even, note also that $1\leq\beta\leq M-1$. Applying $\sqrt{-\Delta}$ to both sides of (\ref{eq:X,Y}), 
	\begin{equation}\label{eq:odd}
	\left|(-\Delta)^{\beta/2}(\phi_\lambda u)\right| = \left|\sqrt{-\Delta}(-\Delta)^n(\phi_\lambda u)\right| \leq \sum_{\substack{X\in \mathscr{U}_n,\\Y\in \mathscr{V}_n}}|\kappa_{X,Y}|\left|\sqrt{-\Delta}\big(X\phi_{-\lambda}Yu\big)\right|.
	\end{equation}
The families $\mathscr{U}_n$ and $\mathscr{V}_n$ now each consist of differential operators of degree at most $2n=\beta-1$. 

Now, $-\sqrt{-\Delta}$ is the infinitesimal generator of the process obtained by subordinating Brownian motion on $G/K$ by the standard $\frac{1}{2}$-stable subordinator on $\mathbb{R}$. By standard subordination theory (see \citet{AppLieProb} \S5.7, pp.~154) $\sqrt{-\Delta}$ may be expressed as a Bochner integral
	\begin{equation}\label{eq:√Δ}
	\sqrt{-\Delta} = \frac{1}{2\sqrt\pi}\int_{0+}^\infty t^{-3/2}(1-T_t)dt,
	\end{equation} 
where $(T_t,t\geq 0)$ denotes the heat semigroup generated by $\Delta$. 

Given $X\in\mathscr{U}_n$, $Y\in\mathscr{V}_n$ and $\sigma\in G$,
	\begin{equation}\label{eq:int_0^1+int_1^infty}
	\begin{aligned}
	\left|\sqrt{-\Delta}(X\phi_{-\lambda}Yu)(\sigma)\right| &= \frac{1}{2\sqrt\pi}\left|\int_{0+}^\infty t^{-3/2}(1-T_t)\big(X\phi_{-\lambda}Yu\big)dt\right| \\
	&\leq \frac{1}{2\sqrt\pi}\Bigg[\left|\int_{0+}^1 t^{-3/2}(1-T_t)\big(X\phi_{-\lambda}Yu\big)(\sigma)dt\right| \\
	&\hspace{90pt}+ \left|\int_1^\infty t^{-3/2}(1-T_t)\big(X\phi_{-\lambda}Yu\big)(\sigma)dt\right|\Bigg].
	\end{aligned}
	\end{equation}
Let $(h_t,t\geq 0)$ denote the heat kernel associated with $(T_t,t\geq 0)$. For the $\int_1^\infty$ term of (\ref{eq:int_0^1+int_1^infty}), note that $\int_1^\infty t^{-3/2}dt=2$, and so
	\begin{align*}
	&\left|\int_1^\infty t^{-3/2}(1-T_t)\big(X\phi_{-\lambda}Yu\big)(\sigma)dt\right| \\
	&\hspace{100pt}= \left|\int_1^\infty t^{-3/2}X\phi_{-\lambda}(\sigma)Yu(\sigma)dt - \int_1^\infty t^{-3/2}T_t\big(X\phi_{-\lambda}Yu\big)(\sigma)dt\right| \\
	&\hspace{100pt}\leq 2|X\phi_{-\lambda}(\sigma)||Yu(\sigma)| + \left|\int_1^\infty t^{-3/2}\int_GX\phi_{-\lambda}(\sigma\tau)Yu(\sigma\tau)h_t(\tau)d\tau dt\right|.
	\end{align*}
By Lemma \ref{lem:Xϕ} and the fact that $\deg X\leq\beta-1$, 
	\begin{equation}\label{eq:C}
	|X\phi_{-\lambda}| \leq C_X\langle\lambda\rangle^{\deg X} \leq C\langle\lambda\rangle^{\beta-1}, \hspace{20pt} \forall\lambda\in\a*,
	\end{equation}
where $C_X$ is as in (\ref{eq:C_X}), and $C=\max\{C_X:X\in\mathscr{U}_n\}$. Thus
	\begin{align*}
	&\left|\int_1^\infty t^{-3/2}(1-T_t)\big(X\phi_{-\lambda}Yu\big)(\sigma)dt\right| \\
	&\hspace{80pt}\leq 2|X\phi_{-\lambda}(\sigma)||Yu(\sigma)| + \int_1^\infty t^{-3/2}\int_G|X\phi_{-\lambda}(\sigma\tau)||Yu(\sigma\tau)|h_t(\tau)d\tau dt \\
	&\hspace{80pt}\leq C\left(2|Yu(\sigma)| + \int_1^\infty t^{-3/2}\int_G|Yu(\sigma\tau)|h_t(\tau)d\tau dt\right)\langle\lambda\rangle^{\beta-1} \\
	&\hspace{80pt}= \Phi_{\beta,Y}^{(1)}(\sigma)\langle\lambda\rangle^{\beta-1},
	\end{align*}
where 
	\begin{equation}\label{eq:Phi_beta,Y^(1)}
	\Phi_{\beta,Y}^{(1)} := C\left(2|Yu| + \int_1^\infty t^{-3/2}T_t\big(|Yu|\big)dt\right).
	\end{equation}
Since $\beta-1\leq M$ and $\langle\lambda\rangle\geq 1$ for all $\lambda\in\a*$, it follows that for all $\lambda\in\a*$,
	\begin{equation}\label{eq:int_1^infty'}
	\left|\int_1^\infty t^{-3/2}(1-T_t)\big(X\phi_{-\lambda}Yu\big)dt\right| \leq \Phi_{\beta,Y}^{(1)}\langle\lambda\rangle^M.
	\end{equation}

We claim that $\Phi_{\beta,Y}^{(1)}\in L^1(K|G|K)$. Clearly $|Yu|\in L^1(K|G|K)$, since it is a continuous function of compact support. Each of the operators $T_t$ is a positivity preserving contraction of $L^1(K|G|K)$, and so
	\[\int_1^\infty t^{-3/2}\int_GT_t\big(|Yu|\big)(\sigma)d\sigma dt = \int_1^\infty t^{-3/2}\left\Vert T_t\big(|Yu|\big)\right\Vert_1 dt \leq \int_1^\infty t^{-3/2}\Vert Yu\Vert_1dt = 2\Vert Yu\Vert_1.\]
By Fubini's theorem, $\int_1^\infty t^{-3/2}T_t\big(|Yu|\big)dt\in L^1(K|G|K)$, with
	\[\left\Vert\int_1^\infty t^{-3/2}T_t\big(|Yu|\big)dt\right\Vert_{L_1(K|G|K)} \leq 2\Vert Yu\Vert_1.\]
It follows by (\ref{eq:Phi_beta,Y^(1)}) that $\Phi_{\beta,Y}^{(1)}\in L^1(K|G|K)$, and that
	\begin{equation}\label{eq:(1)}
	\Vert\Phi_{\beta,Y}^{(1)}\Vert_1 \leq 4C\Vert Yu\Vert_1.
	\end{equation}

For the $\int_{0+}^1$ term of (\ref{eq:int_0^1+int_1^infty}), observe that by Lemma 6.1.12 of \citet{davies2}, pp.~169, as well as the Fubini theorem, 
	\begin{align*}
	\int_{0+}^1t^{-3/2}(1-T_t)\big(X\phi_{-\lambda}Yu\big)dt &= -\int_{0+}^1t^{-3/2}\int_0^t T_s\Delta\big(X\phi_{-\lambda}Yu\big)dsdt \\
	&= -\int_0^1\int_s^1t^{-3/2}T_s\Delta\big(X\phi_{-\lambda}Yu\big)dtds \\
	&= -\int_0^12(s^{-1/2}-1)T_s\Delta\big(X\phi_{-\lambda}Yu\big)ds.
	\end{align*}
Hence, using the product formula for $\Delta$,	
	\begin{equation}\label{eq:int_0^1}
	\begin{aligned}
	\int_{0+}^1t^{-3/2}(1-T_t)\big(X\phi_{-\lambda}Yu\big)dt &= -2\int_0^1(s^{-1/2}-1)\Big\{T_s\big(X\phi_{-\lambda}\Delta Yu\big) \\
	&\hspace{6pt}+ 2\sum_{j=1}^dT_s\big(X_jX\phi_{-\lambda}X_jYu\big) + T_s\big(\Delta X\phi_{-\lambda}Yu\big)\Big\}ds.
	\end{aligned} 
	\end{equation}
Let $C$ and $C_X$ be as in (\ref{eq:C}). Then for all $\sigma\in G$,
	\begin{align*}
	\left|T_s\big(X\phi_{-\lambda}\Delta Yu\big)(\sigma)\right| &= \left|\int_GX\phi_{-\lambda}(\sigma\tau)\Delta Yu(\sigma\tau)h_s(\tau)d\tau\right| \\
	&\leq \int_G|X\phi_{-\lambda}(\sigma\tau)||\Delta Yu(\sigma\tau)|h_s(\tau)d\tau \\
	&\leq C_X\langle\lambda\rangle^{\deg X}\int_G|\Delta Yu(\sigma\tau)|h_s(\tau)d\tau \leq C\langle\lambda\rangle^{\beta-1}T_s|\Delta Yu|(\sigma).
	\end{align*}
In exactly the same way, for $j=1,\ldots,d$,
	\[\left|T_s\big(X_jX\phi_{-\lambda}X_jYu\big)\right| \leq C_X^{(j)}\langle\lambda\rangle^{\deg X+1}T_s|X_jYu| \leq C'\langle\lambda\rangle^{\beta}T_s|X_jYu|,\]
and also
	\[\left|T_s\big(\Delta X\phi_{-\lambda}Yu\big)\right| \leq C_X^{(0)}\langle\lambda\rangle^{\deg X+2}T_s|Yu| \leq C'\langle\lambda\rangle^{\beta+1}T_s|Yu|,\]
where the constants $C_X^{(j)},C^{(j)}$ are chosen so that for all $\lambda\in\a*$ and $j=1,\ldots,d$,
	\[|X\phi_{-\lambda}| \leq C_X^{(0)}\langle\lambda\rangle^{\deg X +2}, \hspace{40pt} |X_jX\phi_{-\lambda}| \leq C_X^{(j)}\langle\lambda\rangle^{\deg X +1},\]
and $C':=\max\{C_X^{(j)}:X\in\mathscr{U}_n,j=0,1,\ldots,d\}$. Such constants exist by Lemma \ref{lem:Xϕ}. Now,
	\[\langle\lambda\rangle^{\beta-1} \leq \langle\lambda\rangle^\beta \leq \langle\lambda\rangle^{\beta+1}\]
for all $\lambda\in\a*$, and hence by (\ref{eq:int_0^1}),
	\begin{align*}
	&\left|\int_{0+}^1t^{-3/2}(1-T_t)\big(X\phi_{-\lambda}Yu\big)dt\right| \\
	&\hspace{20pt}\leq 2\int_0^1(s^{-1/2}-1)\Big\{\left|T_s\big(X\phi_{-\lambda}\Delta Yu\big)\right| + 2\sum_{j=1}^d\left|T_s\big(X_jX\phi_{-\lambda}X_jYu\big)\right| \\
	&\hspace{240pt}+ \left|T_s\big(\Delta X\phi_{-\lambda}Yu\big)\right|\Big\}ds \\
	&\hspace{20pt}\leq 2C'\langle\lambda\rangle^{\beta+1}\int_0^1(s^{-1/2}-1)T_s\left(|\Delta Yu|+2\sum_{j=1}^d|X_jYu|+|Yu|\right)ds.
	\end{align*}
Since $\beta\leq M-1$, it follows that for all $X\in\mathscr{U}_n$ and $Y\in\mathscr{V}_n$,
	\begin{equation}\label{eq:int_0^1'}
	\left|\int_{0+}^1t^{-3/2}(1-T_t)\big(X\phi_{-\lambda}Yu\big)dt\right| \leq \Phi_{\beta,Y}^{(2)}\langle\lambda\rangle^M,
	\end{equation}
where 
	\begin{equation}\label{eq:Phi_beta,Y^(2)}
	\Phi_{\beta,Y}^{(2)} = C'\int_0^1(s^{-1/2}-1)T_s\left(|\Delta Yu|+2\sum_{j=1}^d|X_jYu|+|Yu|\right)ds.
	\end{equation}
Observe that $\Phi_{\beta,Y}^{(2)}\in L^1(K|G|K)$ for all $Y\in\mathscr{V}_n$. Indeed, $u\in C_c^M(K|G|K)$, and $\deg Y\leq \beta-1\leq M-2$, hence $|\Delta Yu|$, $\sum_{j=1}^d|X_jYu|$ and $|Yu|$ are all continuous functions of compact support. Thus $T_s\big(|\Delta Yu|+2\sum_{j=1}^d|X_jYu|+|Yu|\big)\in L^1(K|G|K)$, and, since $T_s$ is an $L^1(K|G|K)$-contraction,
	\[\left\Vert T_s\left(|\Delta Yu|+2\sum_{j=1}^d|X_jYu|+|Yu|\right)\right\Vert_1  \leq \Vert\Delta Yu\Vert_1 +2\sum_{j=1}^d\Vert X_jYu\Vert_1 +\Vert Yu\Vert_1.\]
Noting that $\int_0^1(s^{-1/2}-1)ds = 1$, it follows by Fubini's theorem that  $\Phi_{\beta,Y}^{(2)}\in L^1(K|G|K)$, with
	\begin{equation}\label{eq:(2)}
	\left\Vert\Phi_{\beta,Y}^{(2)}\right\Vert_1 \leq C_X'\left(\Vert\Delta Yu\Vert_1 + 2\sum_{j=1}^d\Vert X_jYu\Vert_1 +\Vert Yu\Vert_1\right).
	\end{equation}

Substituting (\ref{eq:int_0^1'}) and (\ref{eq:int_1^infty'}) into (\ref{eq:int_0^1+int_1^infty}), we obtain the pointwise estimate 
	\begin{equation}\label{eq:int_0^infty}
	\left|\sqrt{-\Delta}(X\phi_{-\lambda}Yu)\right| \leq \frac{1}{2\sqrt\pi}\left(\Phi_{\beta,Y}^{(1)}+\Phi_{\beta,Y}^{(2)}\right)\langle\lambda\rangle^M,
	\end{equation}
for all $X\in\mathscr{U}_n$, $Y\in\mathscr{V}_n$ and $\lambda\in\a*$, where the $\Phi_{\beta,Y}^{(j)}$ ($j=1,2$) are given by (\ref{eq:Phi_beta,Y^(1)}) and (\ref{eq:Phi_beta,Y^(2)}). Hence by (\ref{eq:odd}), for all $\lambda\in\a*$,
	\[\left|(-\Delta)^{\beta/2}(\phi_\lambda u)\right| \leq \sum_{\substack{X\in \mathscr{U}_n,\\Y\in \mathscr{V}_n}}|\kappa_{X,Y}|\left|\sqrt{-\Delta}\big(X\phi_{-\lambda}Yu\big)\right| \leq \Phi_\beta\langle\lambda\rangle^M,\]
where 
	\begin{equation}\label{eq:Phi_odd}
	\Phi_\beta:=\frac{1}{2\sqrt\pi}\sum_{\substack{X\in \mathscr{U}_n,\\Y\in \mathscr{V}_n}}|\kappa_{X,Y}|\left(\Phi_{\beta,Y}^{(1)}+\Phi_{\beta,Y}^{(2)}\right),
	\end{equation}
and $\beta$ is still assumed to be odd. As already noted, $\Phi_{\beta,Y}^{(1)},\Phi_{\beta,Y}^{(2)}\in L^1(K|G|K)$ for all $Y\in\mathscr{V}_n$, and hence $\Phi_\beta\in L^1(K|G|K)$. Moreover, by (\ref{eq:(1)}) and (\ref{eq:(2)}), 
	\begin{equation}\label{eq:L1odd}
	\begin{aligned}
	\left\Vert\Phi_\beta\right\Vert_1 &\leq \kappa_\beta'\sum_{Y\in\mathscr{V}_n}\Bigg(\Vert\Delta Yu\Vert_1+\sum_{j=1}^d\Vert X_jYu\Vert_1 +\Vert Yu\Vert_1\Bigg),
	\end{aligned}
	\end{equation}
for some positive constant $\kappa_\beta'$. In particular, we have verified (\ref{eq:A.2.M'}) when $\beta$ is odd.
\end{proof}

\begin{cor}
Let $q:G\times\a*\to\mathbb{R}$ be of the form (\ref{eq:qEG}). Then for $\kappa$ sufficiently large, the conditions of Corollary \ref{cor:sub-Mark} are satisfied. In particular, $-q(\sigma,D)$ extends to the infinitesimal generator of a strongly continuous sub-Feller semigroup on $C_0(K|G|K)$.
\end{cor}

\section{Proof of Proposition \ref{prop:Fmoll}}\label{sec:PfofProp}
\begin{enumerate}
\item Let $0<\epsilon\leq 1$ and $\lambda\in\a*$. Using a change of variable $H\mapsto\epsilon^{-1}H$,
	\begin{align*}
	\hat{j_\epsilon}(\lambda) = \mathscr{F}(l_\epsilon)(\lambda) &= \int_{\al}e^{-i\lambda(H)}\epsilon^{-m}l(\epsilon^{-1}H)dH \\
	&= \int_{\al}e^{-i\epsilon\lambda(H)}l(H)dH = \mathscr{F}(l)(\epsilon\lambda) = \hat{j}(\epsilon\lambda).
	\end{align*}
\item The map $l$ is symmetric under $H\mapsto -H$, and hence $\mathscr{F}(l)=\hat j$ is real-valued. Therefore, given $u,v\in L^2(K|G|K)$ and $0<\epsilon\leq 1$,
	\[\langle J_\epsilon u,v\rangle = \int_{\a*}\hat{j}(\epsilon\lambda)\hat{u}(\lambda)\overline{\hat{v}(\lambda)}\omega(d\lambda) = \int_{\a*}\hat{u}(\lambda)\overline{\hat{j}(\epsilon\lambda)\hat{v}(\lambda)}\omega(d\lambda) = \langle u, J_\epsilon v\rangle.\]
To see that $J_\epsilon$ is a contraction, note that $|\hat{j_\epsilon}(\lambda)|=|\hat j(\epsilon\lambda)|\leq \hat{j}(0)= 1$ for all $\lambda\in\a*$, and so by Plancherel's identity
	\[\Vert J_\epsilon u\Vert = \Vert\hat j_\epsilon\hat u\Vert_{L^2(\a*,\omega)} \leq \Vert\hat u\Vert_{L^2(K|G|K)} = \Vert u \Vert,\]
for all $u\in L^2(K|G|K)$ and all $0<\epsilon\leq 1$.
\item Let $s\geq 0$ and $0<\epsilon\leq 1$. By Theorem \ref{thm:ComDiag}, $\hat{j}\in\mathcal{S}(\a*)^W$,
and hence there is $\kappa>0$ such that
	\[\langle\lambda\rangle^s\left|\hat{j}(\epsilon\lambda)\right| \leq \kappa, \hspace{20pt} \forall \lambda\in\a*.\] 
Then, using Proposition \ref{prop:negdef} (\ref{c_phi}), 
	\[\Psi(\lambda)^s\left|\hat{j}(\epsilon\lambda)\right| \leq c_\psi^{s/2}\langle\lambda\rangle^s\left|\hat{j}(\epsilon\lambda)\right|\leq c_\psi^{s/2}\kappa,\]
for all $\lambda\in\a*$. Let $u\in L^2(K|G|K)$. By Plancherel's identity, 
	\[\int_{\a*}\Psi(\lambda)^{2s}\left|\hat{j}(\epsilon\lambda)\right|^2|\hat{u}(\lambda)|^2\omega(d\lambda) \leq c_\psi^s\kappa^2\Vert u \Vert^2 < \infty.\]
By Proposition \ref{prop:Fmoll} (\ref{epslam}), $(J_\epsilon u)^\wedge(\lambda) = \hat{j}(\epsilon\lambda)\hat{u}(\lambda)$, for all $\lambda\in\a*$, and hence 
	\[\int_{\a*}\Psi(\lambda)^{2s}|(J_\epsilon u)^\wedge(\lambda)|^2\omega(d\lambda)<\infty.\]
That is, $J_\epsilon u\in H^{\psi,s}$. 

Next, suppose $u\in H^{\psi,s}$. Then, since $|\hat{j_\epsilon}|\leq 1$,
	\[\Vert J_\epsilon u\Vert_{\psi,s} = \Vert\Psi^s\hat{j_\epsilon}\hat{u}\Vert_{L^2(\a*,\omega)} \leq \Vert\Psi^s\hat{u}\Vert_{L^2(\a*,\omega)} = \Vert u\Vert_{\psi,s},\]
as desired.
\item By Theorem 1 on page 250 of \citet{evans}, for all $v\in \mathcal{S}(\al)^W$, $l_\epsilon\ast v\rightarrow v$ as $\epsilon\rightarrow 0$, in the classical Sobolev space $W^s(\a*)$, for all $s\geq 0$.  Therefore,
	\[\lim_{\epsilon\rightarrow 0}\int_{\a*}(1+|\lambda|^2)^s|\mathscr{F}(l_\epsilon\ast v-v)(\lambda)|^2d\lambda = 0, \hspace{20pt} \forall s\geq 0,\;v\in \mathcal{S}(\al)^W.\]
Let $u\in C_c^\infty(K|G|K)$ and $v=\mathscr{F}^{-1}(\hat{u})$. Then $v\in \mathcal{S}(\al)^W$, and
	\[\mathscr{F}(l_\epsilon\ast v - v) = (\hat{j}_\epsilon-1)\hat{u} = (J_\epsilon u - u)^\wedge.\]
Hence $\lim_{\epsilon\rightarrow 0}\int_{\a*}(1+|\lambda|^2)^s|(J_\epsilon u - u)^\wedge(\lambda)|^2d\lambda = 0$, for all $s\geq 0$. By (\ref{eq:hcc-bnd}), 
	\begin{align*}
	\int_{\a*}(1+|\lambda|^2)^s|(J_\epsilon u &- u)^\wedge(\lambda)|^2\omega(d\lambda) \\
	&\leq \int_{\a*}(1+|\lambda|^2)^s|(J_\epsilon u - u)^\wedge(\lambda)|^2(C_1+C_2|\lambda|^p)^2d\lambda \\
	&\leq \kappa\int_{\a*}(1+|\lambda|^2)^{s+p}|(J_\epsilon u - u)^\wedge(\lambda)|^2d\lambda,
	\end{align*}
for some constant $\kappa>0$, and where $p=\frac{\dim N}{2}$. Thus
	\[\lim_{\epsilon\rightarrow 0}\int_{\a*}(1+|\lambda|^2)^s|(J_\epsilon u - u)^\wedge(\lambda)|^2\omega(d\lambda) = 0, \hspace{20pt} \forall s\geq 0.\]
By Proposition \ref{prop:negdef} (\ref{c_phi}),
	\begin{align*}
	\Vert J_\epsilon u-u\Vert_{\psi,s}^2 &= \int_{\a*}(1+\psi(\lambda))^s|(J_\epsilon u - u)^\wedge(\lambda)|^2\omega(d\lambda) \\
	&\leq c_\psi\int_{\a*}(1+|\lambda|^2)^s|(J_\epsilon u - u)^\wedge(\lambda)|^2\omega(d\lambda) \rightarrow 0
	\end{align*}
as $\epsilon\rightarrow 0$. Since $C_c^\infty(K|G|K)$ is dense in $H^{\psi,s}$, Proposition \ref{prop:Fmoll} (\ref{J_epsu-u}) follows.
\end{enumerate}

\begin{ack}
Many thanks to David Applebaum for his advice and support with writing this paper. Thanks also to the University of Sheffield's School of Mathematics and Statistics, and to the EPSRC for providing PhD funding while this research was carried out.
\end{ack}
\bibliographystyle{plainnat}
\bibliography{mybib}
\end{document}